\documentclass[11pt]{amsart}
\usepackage{graphicx}
\usepackage{tikz}
\usepackage{latexsym,color,amsmath,systeme, amsthm,amssymb,amscd,amsfonts}
\usepackage{graphicx}

\usepackage{hyperref}
\usepackage{pgf,tikz,pgfplots, pgfplotstable}
\pgfplotsset{compat=1.14}
\usepackage{float}
\usepackage{mathrsfs}
\usetikzlibrary{arrows}
\usepackage[style=english]{csquotes}
\usepackage{enumerate}
\usepackage{lscape}
\usepackage[width=\textwidth]{caption}
\usepackage{subcaption}
\usepackage{amsmath}
\usepackage{amsfonts}
\usepackage{amssymb}
\usepackage{pgf,tikz}
\usepackage{tkz-tab}
\usepackage{pgfplots}
\usepackage{siunitx}
\usetikzlibrary{angles, quotes}

\usetikzlibrary{decorations.pathmorphing}
\tikzset{discont/.style={decoration={zigzag,segment length=12pt, amplitude=4pt},decorate}}
\pgfplotsset{compat=1.11}
\usetikzlibrary{intersections,fillbetween}
 \usetikzlibrary{decorations.text}
\usetikzlibrary{shapes,snakes,arrows,intersections, backgrounds}
\usetikzlibrary{scopes,svg.path,shapes.geometric,shadows}
\DeclareCaptionSubType*[arabic]{figure}
\tikzset{
	v1/.style={line width=.5pt,blue!33!black},
	v2/.style={line width=.5pt,blue!66!black},
	v3/.style={line width=.5pt,blue!33},
	v4/.style={line width=.5pt,blue!66},
	v5/.style={line width=.5pt,black}
}

\definecolor{dred}{HTML}{C11B17}
\definecolor{dgreen}{HTML}{41A317}
\definecolor{dblue}{HTML}{00008B}
\definecolor{niceblue}{HTML}{00008B}
\definecolor{brilliantrose}{rgb}{1.0, 0.33, 0.64}
\definecolor{gold}{HTML}{ffd700}
\definecolor{lgold}{HTML}{ffe140}

\newcommand{\cemph}[1]{\textcolor{dred}{\emph{#1}}}

\newcommand{\R}{\mathbb R}
\newcommand{\K}{\mathcal K}

\newcommand{\B}{\mathbb{B}}
\renewcommand{\S}{\mathbb{S}}

\newcommand{\relint}{\mathrm{relint}}

\newcommand{\conv}{\mathrm{conv}}
\newcommand{\bd}{\mathrm{bd}}
\newcommand{\pos}{\mathrm{pos}}
\newcommand{\aff}{\mathrm{aff}}

\newcommand{\inter}{\mathrm{int}}

\definecolor{zzttqq}{rgb}{0.6,0.2,0}
			\definecolor{ccqqqq}{rgb}{0.8,0,0}
\definecolor{ffvvqq}{rgb}{1,0.3333333333333333,0}
\definecolor{zzffqq}{rgb}{0.6,1,0}
\definecolor{qqwuqq}{rgb}{0,0.39215686274509803,0}
\definecolor{ffzzqq}{rgb}{1,0.6,0}
\definecolor{ffqqqq}{rgb}{1,0,0}
\definecolor{zzttqq}{rgb}{0.6,0.2,0}
\definecolor{uuuuuu}{rgb}{0.26666666666666666,0.26666666666666666,0.26666666666666666}

\newtheorem{thm}{Theorem}[section]
\newtheorem{lemma}[thm]{Lemma}
\newtheorem{remark}[thm]{Remark}
\newtheorem{proposition}[thm]{Proposition}
\newtheorem{cor}[thm]{Corollary}

\newtheorem{example}[thm]{Example}

\usetikzlibrary{calc,intersections}
\begin{document}

\title[Mean inequalities for symmetrizations of convex sets]{Tightening and reversing the arithmetic-harmonic mean inequality for symmetrizations of convex sets}

\author[R. Brandenberg]{Ren\'e Brandenberg}\email{brandenb@ma.tum.de}
\author[K. von Dichter]{Katherina von Dichter}\email{dichter@ma.tum.de}
\author[B. Gonz\'alez Merino]{Bernardo Gonz\'alez Merino}\email{bgmerino@um.es}

\begin{abstract}
This paper deals with four symmetrizations of a convex set $C$: the intersection, the harmonic and the arithmetic mean, and the convex hull of $C$ and $-C$. A well-known result of Firey shows that those means build up a subset-chain in 
the given order. On the one hand, we determine the dilatation factors, depending on the asymmetry of $C$, to reverse 
the containments between any of those symmetrizations. On the other hand, we tighten the relations proven by Firey and show a stability result concerning those factors near the simplex.


\end{abstract}

\thanks{This research is a result of the activity developed within the framework of the Programme in Support of Excellence Groups of the Regi\'on de Murcia, Spain, by Fundaci\'on S\'eneca, Science and Technology Agency of the Regi\'on de Murcia. 
The third author is partially supported by Fundaci\'on S\'eneca project 19901/GERM/15 Spain and by Grant PGC2018-094215-B-I00 funded by MCIN/AEI/10.13039/501100011033 and by ``ERDF A way of making Europe''}

\date{\today}\maketitle
\section{Introduction and Notation}\label{sec:IntrNotMain}

The arithmetic-geometric-harmonic mean inequality together with minimum and maximum (which can be seen as the extreme means) states in the two-argument case 
\begin{equation}\label{eq:means_of_numbers}
\min\{a,b\}\leq\left(\frac{a^{-1}+b^{-1}}{2}\right)^{-1}\leq \sqrt{ab}\leq \frac{a+b}{2}\leq\max\{a,b\}
\end{equation}
for any $a, b > 0$, with equality in any of the inequalities if and only if $a=b$ (see \cite{HLP,Sch}).

For any $X\subset\R^n$ let $\conv(X)$ denote the \cemph{convex hull}, i.e.,~the smallest convex set containing $X$. A \cemph{segment} is the convex hull of $\{x,y\} \subset \R^n$, which we abbreviate by $[x,y]$. 
For any $X,Y \subset\R^n$, $\rho \in \R$ let $X+Y =\{x+y:x\in X,y\in Y\}$ be the \cemph{Minkowski sum} of $X$ and $Y$, and  $\rho X= \{ \rho x: x \in X\}$ the \cemph{$\rho$-dilatation} of $X$. We abbreviate $(-1)X$ by $-X$. The family of all \cemph{convex bodies} (full-dimensional compact convex sets) is denoted by $\K^n$ and for any $C\in\K^n$ we write $C^\circ=\{a\in\R^n: a^T x \leq 1,\, x\in C\}$ for the \cemph{polar} of $C$.

All the means above can be generalized for convex sets. One may identify means of numbers by means of segments via associating $a, b > 0$ with $[-a,a]$ and $[-b,b]$. Thus, e.g., the arithmetic mean of $a$ and $b$ is identified with $[-\frac{1}{2} \left( a+b \right), \frac{1}{2} \left( a+b \right) ] = \frac{1}{2} \left( [-a,a]+[-b,b] \right)$. In general, the \cemph{arithmetic mean} of $K,C \in \K^n$ is defined by $\frac{1}{2} (K+C)$, the \cemph{minimum} by $K\cap C$, and the \cemph{maximum} by $\conv(K\cup C)$. 
Since polarity can be regarded as the higher-dimensional counterpart of the inversion operation $x\rightarrow 1/x$ (cf.~\cite{MR}), the \cemph{harmonic mean} of $K$ and $C$ is defined by $\left( \frac{1}{2}(K^\circ+C^\circ) \right)^{\circ}$.
The geometric mean has been extended in several ways (cf.~\cite{BLYZ} or \cite{MR}). It would need a separate, more involved treatment. Here we focus only on the four other means. The study of means of convex bodies has been started by Firey in the 1960's  \cite{F,F2,F3}, but there also exist several recent papers (see, e.g.,~\cite{MR,MR2,MMR}).

Perphaps the most essential result of Firey is the extension of the harmonic-arithmetic mean inequality from positive numbers to convex bodies with 0 in their interior in \cite{F} (see \cite{MR} for a nice and short proof). Moreover, Firey's inequality may again be extended involving minimum and maximum.

\begin{proposition} \label{prop:means_of_sets}
 Let $C,K \in\K^n$ with $0$ in their interior. Then
 \begin{equation}\label{eq:means_of_sets}
   K\cap C\subset \left(\frac{K^\circ+C^\circ}{2}\right)^{\circ}\subset\frac{K+C}{2}\subset\conv(K\cup C),
 \end{equation}
 with equality between any of the means if and only if $K=C$.
\end{proposition}

In the following we analyze sharpness of the set-containment inequalities with respect to optimal containment (instead of equality of sets):
For any $C, K\in\K^n$ we say $K$ is \cemph{optimally contained} in $C$ and denote it by $K\subset^{opt}C$, if $K\subset C$ and $K\not \subset \rho C+t$ for any $\rho \in [0,1)$ and $t\in\R^n$. For $C_1, \dots, C_k  \in\K^n$ we say $C_1 \subset \ldots \subset C_k$ is \cemph{left-to-right optimal} if $C_1 \subset^{opt} C_k$. 

The starting point of our investigation is the following generalization of
\cite[Theorem 3]{BDG} for arbitrary convex sets with 0 in their interior. 
\begin{thm}\label{thm:Charact_Opt_Means_KC}

	Let $C,K \in \K^n$ with $0 \in \inter(K \cap C)$. 
	Then 
	\[K \cap C \subset^{opt} \conv(K \cup C) \iff \left(\frac12 (K^\circ+C^\circ) \right)^{\circ}\subset^{opt} \frac12 (K+C).\]
\end{thm}

Note that Theorem \ref{thm:Charact_Opt_Means_KC} implies that left-to-right optimality in \eqref{eq:means_of_sets} depend solely on the optimal containment of the harmonic in the arithmetic mean.
 
If $C=-C+t$ for some $t \in \R^n$, we say $C$ is \cemph{symmetric}, and if $C=-C$, we say $C$ is \cemph{$0$-symmetric}.
The family of 0-symmetric convex bodies is denoted by $\K^n_0$. 

A special focus in our study lies on optimal containments of means of $C$ and $-C$ of a convex body $C$, which are all symmetrizations of $C$.
Symmetrizations are frequently used in convex geometry, e.g.,~as extreme cases of a variety of geometric inequalities. Consider, e.g., the Bohnenblust inequality \cite{Bo}, which bounds the ratio of the circumradius and the diameter of convex bodies in arbitrary normed spaces. The equality case in this inequality is reached in normed spaces with $S \cap (-S)$ or $ \frac{1}{2} (S-S)$ as their unit balls \cite{BrK}, where $S$ denotes an $n$-simplex with center of gravity in 0. These means also appear in characterizations of spaces for which $C$ is complete or reduced 
\cite[Prop.~3.5 -- 3.10]{BGJM}. 
We provide more motivation on considering optimal containments between different symmetrizations of $C$ in the Appendix. 

A major part of this paper is devoted to a better understanding of the optimal containments between those symmetrizations depending on the asymmetry of the initial body. We naturally require all symmetrizations of an already symmetric $C$ to coincide with $C$. This is always true for the arithmetic mean $\frac12(C-C)$, but $0$ needs to be the center of symmetry for the other three considered means. This indicates the need of fixing a meaningful center for every convex body. The most common choice of an asymmetry measure and a corresponding center are the \cemph{Minkowski asymmetry} of $C \in \K^n$, which is defined by 
\[s(C):=\inf \{ \rho >0: C-c \subset \rho (C-c), c \in \R^n \},\] and the (not necessarily unique) \cemph{Minkowski center} of $C$, which is any $c \in \R^n$ fulfilling $C-c \subset s(C)(c-C)$ \cite{Gr, BG}. If $0$ is a Minkowski center, we say $C$ is \cemph{Minkowski centered}. 

Note that $s(C)\in[1,n]$, with $s(C)=1$ if and only if $C$ is symmetric, and $s(C)=n$ if and only if $C$ is an $n$-dimensional simplex \cite{Gr}. Moreover, the Minkowski asymmetry $s:\K^n\rightarrow[1,n]$ is continuous w.r.t.~the Hausdorff metric (see \cite{Gr}, \cite{Sch} for some basic properties) and invariant under non-singular affine transformations. 
We believe that the Minkowski asymmetry is most suitable for studying optimal containments and consequently focus on Minkowski centered convex sets.

The classical norm relations $\|x\|_\infty \leq \|x\|_2\leq \|x\|_1$ with $x\in\R^n$ can be naturally reversed 
by the inequalities $\|x\|_1 \leq \sqrt{n}\|x\|_2 \leq n \|x\|_\infty$, which both transfer to left-to-right optimal containments between the corresponding unit ball of these $\ell_p$-spaces. Similarly, we consider the norms induced by the means of $K$ and $C$. Doing so, \eqref{eq:means_of_sets} can be read as follows:
\begin{equation}\label{eq:normrelations}
\|x\|_{\conv(K\cup C)} \leq \|x\|_{\frac{K+C}{2}} \leq \|x\|_{\left(\frac{K^\circ+C^\circ}{2}\right)^{\circ}} \leq \|x\|_{K\cap C}.    
\end{equation}
In order to reverse this chain of inequalities, we need to provide a chain of (optimal) inclusions, which is reverse to \eqref{eq:means_of_sets}. This is not possible for general convex bodies, since the scaling factors of the reverse inclusions cannot be bounded in general. However, assuming Minkowski centeredness of the considered body, this problem can be fixed.

\begin{thm}\label{thm:reverse_inclusions}
Let 
$C\in\K^n$ be Minkowski centered. 
Then
\begin{enumerate}[(i)]
\item $\conv(C\cup(-C))\subset^{opt} s(C) (C\cap(-C))$,
\item $\conv(C\cup(-C))\subset^{opt} \frac{2s(C)}{s(C)+1} \frac{C-C}{2}$,
\item $\left(\frac{C^\circ-C^\circ}{2}\right)^\circ \subset^{opt} \frac{2s(C)}{s(C)+1}(C\cap(-C))$,
\item $\frac{C-C}{2}\subset^{opt} \frac{s(C)+1}{2} (C\cap(-C))$, and
\item $\conv(C\cup(-C))\subset^{opt}  \frac{s(C)+1}{2}\left(\frac{C^\circ-C^\circ}{2}\right)^\circ$.
\item $\frac{C-C}{2}\subset \frac{s(C)+1}{2}\left(\frac{C^\circ-C^\circ}{2}\right)^\circ$, and for all $s \in [n]$ there exists a Minkowski centered $C \in \K^n$ with $s(C)=s$, such that the containment is optimal.
\end{enumerate}
\end{thm}

After the proof of Theorem \ref{thm:reverse_inclusions} we will also provide an example that shows that the containment in Part (vi) above may not be optimal and derive a lower bound for the minimal dilatation factor needed for this covering.

As a consequence of Theorem \ref{thm:reverse_inclusions}, we derive the following left-to-right optimal containment chains.

\begin{cor} \label{rem:left-to-right-opt}
Let 
$C \in \K^n$ be Minkowski centered.
Then the following containment chains are both left-to-right optimal:
\begin{enumerate}[(i)]
    \item $\conv(C \cup (-C)) \subset \frac{s(C)+1}{2} \left(\frac{C^\circ-C^\circ}{2}\right)^\circ \subset s(C) (C \cap (-C))$, and
    \item $\conv(C \cup (-C)) \subset \frac{2s(C)}{s(C)+1}\frac{C-C}{2} \subset s(C) (C \cap (-C))$.
\end{enumerate}
Moreover, for the following containment chains always apply:
\begin{enumerate}[(i)]
    \item[(iii)] $\frac{C-C}{2} \subset \conv(C \cup (-C)) \subset \frac{s(C)+1}{2} \left(\frac{C^\circ-C^\circ}{2}\right)^\circ$, and
    \item[(iv)] $\frac{C-C}{2} \subset \frac{s(C)+1}{2} C \cap (-C) \subset \frac{s(C)+1}{2} \left(\frac{C^\circ-C^\circ}{2}\right)^\circ$, and
\end{enumerate}
for every $s \in [n]$ there exist $C \in \K^n$ with $s(C)=s$, such that these chains are left-to-right optimal.
\end{cor}

Based on this corollary, one obtains, e.g., that the following reverse inequality chain of \eqref{eq:normrelations} is sharp w.r.t.~$s(C)$
\begin{equation}\label{eq:LtoR}
\|x\|_{C\cap (-C)} \leq \frac{s(C)+1}{2} \|x\|_{\frac{C-C}{2}} \leq s(C) \|x\|_{\conv(C\cup (-C))}. \end{equation}

Some containments of symmetrizations in the forward direction are always optimal (see \cite{BDG}): 
\begin{equation*} 
\frac{C-C}{2} \subset^{opt} \conv ( C \cup (-C) ) \quad\text{ and }\quad C \cap (-C) \subset^{opt} \left(\frac{C^\circ-C^\circ}{2}\right)^\circ.
\end{equation*}
Using Proposition \ref{thm:Charact_Opt_Means_KC}, we see that \eqref{eq:means_of_sets} may be left-to-right optimal even for non-symmetric $C$.
In particular, considering a regular Minkowski centered simplex $S \in \K^3$, the four means are a cross-polytope (minimum), a rhombic dodecahedron (harmonic mean), a cube octahedron (arithmetic mean), and a cube (maximum) and they build a left-to-right optimal chain of containments (see Figure \ref{fig:symms-of-tetrahedron}).

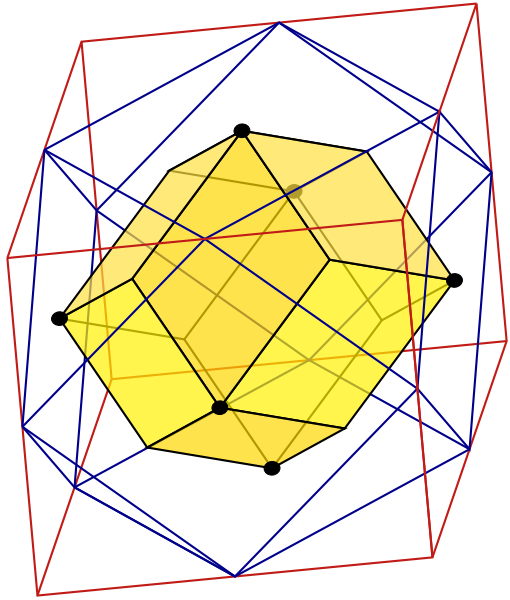
\begin{figure}[ht] 
	\centering
	\begin{tikzpicture}[scale=2.5]

	\draw[thick, dred,rotate around x=12, rotate around y=10](1,1,-1)--(-1,1,-1)--(-1,1,1);
	\draw[thick, dred,rotate around x=12,rotate around y=10](1,1,-1)--(1,-1,-1)--(1,-1,1)--(-1,-1,1)--(-1,1,1);
	\draw[thick, dred, rotate around x=12,rotate around y=10](1,-1,-1)--(-1,-1,-1)--(-1,1,-1);
	\draw[thick, dred, rotate around x=12,rotate around y=10](-1,-1,-1)--(-1,-1,1);

    \draw [fill,rotate around x=12,rotate around y=10] (0,0,-1) circle [radius=0.04];
    
	\draw[thick, niceblue,rotate around x=12,rotate around y=10](-1,1,0)--(-1,0,1);
	\draw[thick, niceblue,rotate around x=12,rotate around y=10](-1,0,1)--(-1,-1,0)--(0,-1,1);
	\draw[thick, niceblue,rotate around x=12,rotate around y=10](-1,-1,0)--(0,-1,1)--(1,-1,0);
	\draw[thick,  niceblue,rotate around x=12,rotate around y=10](-1,-1,0)--(0,-1,-1)--(1,-1,0);
	\draw[thick,  niceblue,rotate around x=12,rotate around y=10](-1,1,0)--(-1,0,-1)--(-1,-1,0);
	\draw[thick,  niceblue,rotate around x=12,rotate around y=10](-1,0,-1)--(0,-1,-1)--(1,0,-1);
	\draw[thick,  niceblue,rotate around x=12,rotate around y=10](-1,0,-1)--(0,1,-1)--(1,0,-1);
	
	\draw[thick, rotate around x=12,rotate around y=10](0,-1,0)--(-1/2,-1/2,-1/2)--(0,0,-1)--(1/2,-1/2,-1/2)--(0,-1,0);
	\draw[thick, rotate around x=12,rotate around y=10](0,1,0)--(-1/2,1/2,-1/2)--(0,0,-1)--(1/2,1/2,-1/2)--(0,1,0);
	\draw[thick, rotate around x=12,rotate around y=10](-1/2,-1/2,-1/2)--(-1,0,0);
	\draw[thick, rotate around x=12,rotate around y=10](1/2,-1/2,-1/2)--(1,0,0);
	\draw[thick, rotate around x=12,fill=yellow, fill opacity=0.7, rotate around y=10](1,0,0)--(1/2,1/2,1/2)--(0,0,1)--(1/2,-1/2,1/2)--(1,0,0);
	\draw[thick, rotate around x=12,fill=gold, fill opacity=0.7, rotate around y=10](0,1,0)--(1/2,1/2,1/2)--(0,0,1)--(-1/2,1/2,1/2)--(0,1,0);
	\draw[thick, rotate around x=12,fill=yellow, fill opacity=0.7, rotate around y=10](-1,0,0)--(-1/2,1/2,1/2)--(0,0,1)--(-1/2,-1/2,1/2)--(-1,0,0);
	\draw[thick, rotate around x=12,fill=gold, fill opacity=0.7, rotate around y=10](0,-1,0)--(1/2,-1/2,1/2)--(0,0,1)--(-1/2,-1/2,1/2)--(0,-1,0);
	\draw[thick, rotate around x=12,fill=lgold, fill opacity=0.7, rotate around y=10](0,1,0)--(1/2,1/2,1/2)--(1,0,0)--(1/2,1/2,-1/2)--(0,1,0);
	\draw[thick, rotate around x=12,fill=lgold, fill opacity=0.7, rotate around y=10](0,1,0)--(-1/2,1/2,1/2)--(-1,0,0)--(-1/2,1/2,-1/2)--(0,1,0);

	\draw[thick, rotate around x=12,niceblue,rotate around y=10](-1,0,1)--(0,-1,1)--(1,0,1)--(0,1,1)--(-1,0,1);
	\draw[thick,rotate around x=12,niceblue,rotate around y=10](-1,1,0)--(0,1,1)--(1,1,0)--(0,1,-1)--(-1,1,0);
	\draw[thick,rotate around x=12,niceblue,rotate around y=10](1,0,1)--(1,-1,0)--(1,0,-1)--(1,1,0)--(1,0,1);

    \draw[thick, rotate around x=12,dred, rotate around y=10](-1,1,1)--(1,1,1)--(1,-1,1);
	\draw[thick, rotate around x=12,dred,rotate around y=10](1,1,1)--(1,-1,1);
	\draw[thick, dred,rotate around x=12,rotate around y=10](1,1,1)--(1,1,-1);

    \draw [fill,rotate around x=12,rotate around y=10] (1,0,0) circle [radius=0.04];
    \draw [fill,rotate around x=12,rotate around y=10] (0,0,1) circle [radius=0.04];
    \draw [fill,rotate around x=12,rotate around y=10] (-1,0,0) circle [radius=0.04];
    \draw [fill,rotate around x=12,rotate around y=10] (0,1,0) circle [radius=0.04];
    \draw [fill,rotate around x=12,rotate around y=10] (0,-1,0) circle [radius=0.04];

	\end{tikzpicture}
	\caption{Symmetrizations of a regular simplex $S \subset \R^3$: Minimum $S \cap (-S)$ is a cross-polytope (convex hull of black points), harmonic mean $\left(\frac{S^\circ-S^\circ}{2}\right)^\circ$ is a rhombic dodecahedron (yellow), arithmetic mean  $\left(\frac{S-S}{2}\right)$ is a cube octahedron (blue), and maximum $\conv(S \cup (-S))$ is a cube (red) .
	}
	\label{fig:symms-of-tetrahedron}
\end{figure}

This property remains true for the four symmetrizations of a regular Minkowski centered simplex in any odd dimension. In contrast, for a regular Minkowski centered simplex $S$ in even dimensions we show in Lemma \ref{lem:Simplex_Odd} that
\[
S \cap (-S) \subset^{opt} \frac{n}{n+1} \conv( S \cup (-S) )\quad\text{and}\quad
\left(\frac{S^\circ-S^\circ}{2}\right)^\circ \subset^{opt}\frac{n(n+2)}{(n+1)^2} \frac{S-S}{2}.
\]

Concerning the above, we proceed with a stability result. First we introduce several parameters which we need throughout the upcoming results.

\begin{align*}
\psi&:=\psi(n,s) := \frac{(n-s+1)(s+1)}{1-n(n-s)(n+s(n+1))} - n, \\ \mu&:=\mu(n,s) = \frac{n+1}{s+1}\left( 1- \frac{s(n+1)(n-s)}{1-n(n-s)} \right), \\
\gamma_1&:=\gamma_1(n) := \frac12(n-1+\sqrt{(n-2)n+5}), \\
\gamma_2&:=\gamma_2(n):=\frac{n^4+n^3+2n^2+\sqrt{n^8+6n^7+17n^6+28n^5+28n^4+12n^3-4n^2-12n-4}}{2(n^3+2n^2+3n+1)},\\
\gamma_3&:=\gamma_3(n):=\frac{n^4+3n^3+2n^2+1+\sqrt{n^8+6n^7+13n^6+8n^5-14n^4-22n^3+8n+1}}{2(n^3+2n^2+2n)} .  
\end{align*}
One can check that $n - \frac{1}{n}  < \gamma_2 < \gamma_3 < n$ and that both $\psi$ and $\mu$ become 1 in case $n=s$. Moreover, we will see that $\psi \frac{n}{n+1} > 1$ for all $s > \gamma_2$, while $\mu \psi \frac{n(n+2)}{(n+1)^2} <1$ for all $s > \gamma_3$.

\begin{thm}\label{thm:minMax_mean_improved}
Let $n$ be even and $C\in\K^n$ be Minkowski centered with $s(C)=s$.
Then 
\begin{enumerate}[(i)]
	\item $\displaystyle C \cap (-C) \subset \psi \, \frac{n}{n+1} \conv(C\cup(-C))$, if $s \ge \gamma_2(n)$, and
	\item $\displaystyle \left(\frac{C^\circ+(-C)^\circ}{2}\right)^{\circ} \subset  \mu \psi \, \frac{n(n+2)}{(n+1)^2}
	\frac{C-C}{2}$, if $s \ge \gamma_3(n)$. 
\end{enumerate}
\end{thm}

One should recognize that the factor $\mu \psi \frac{n(n+2)}{(n+1)^2}$ in Part (ii) of Theorem \ref{thm:minMax_mean_improved} becomes greater than 1 for $s < \gamma_3(n)$. However, from Part (i) of Theorem \ref{thm:minMax_mean_improved} together with Theorem \ref{thm:Charact_Opt_Means_KC} we obtain that the harmonic mean of any pair of Minkowski centered convex bodies $C$ and $-C$ cannot be optimally contained in their arithmetic mean for any $s \in [\gamma_2,\gamma_3]$.



Whenever \eqref{eq:means_of_sets} is left-to-right optimal for some  Minkowski centred convex body $C$ there also exist a series of Minkowski centered convex bodies with any smaller asymmetry providing a left-to-right optimality for the full chain (see Lemma \ref{lem:Asym_Descent_Chain}). Thus we aim to determine the smallest number $\gamma(n) \in [n-1,n]$ such that for every Minkowski centered $C\in\K^n$ with $s(C)\ge \gamma(n)$ the harmonic mean of $C$ and $-C$ is not optimally contained in their arithmetic mean.
We already introduced $\gamma(n)$ in \cite{BDG} as the \cemph{asymmetry threshold of means} and it is shown there that $\gamma(2) = \frac{1+\sqrt{5}}{2}=: \varphi  $ is the golden ratio, while $\gamma(n)=n$ whenever $n$ is odd.
Here we present a result on the asymmetry threshold for arbitrary even dimensions.

\begin{thm}\label{thm:gamma} Let $n$ be even.  
Then 
\begin{equation*}
n-1 < \gamma_1 \leq \gamma(n) \leq \gamma_2<n.
\end{equation*}
\end{thm}

One may recognize the following: it is well-known that the golden ratio, which is also $\gamma(2)$, can be obtained from solving the equation $\frac{a+b}{a} = \frac{a}{b}$ for $a>b>0$. However, one can similarily obtain the values of $\gamma_1$ in Theorem \ref{thm:gamma} from solving the equation $\frac{(n-1)a+b}{a}=\frac{a}{b}$ and therefore consider the values of $\gamma_1$ as a generalized golden ratio.

The asymmetry threshold provides us with a lower bound for the values of $s$ such that \eqref{eq:means_of_sets} cannot be left-to-right optimal. In the following we want to go one step further and determine the possible values for the contraction factors $\alpha(s)$ and $\beta(s)$ for which the minimum is optimally contained in the according contraction of the maximum and for which the harmonic mean is optimally contained in the contraction of the arithmetic mean, respectively.

\begin{thm}\label{thm:small_asym_no_improve}
Let $C \in\K^n$ be Minkowski centered with $s(C)=s$. 
\begin{enumerate}[a)]
\item Let $\alpha(s) \in \R$ such that $C  \cap (-C) \subset^{opt} \alpha(s) \, \conv(C \cup (-C))$ and $\alpha_1(s)$, $\alpha_2(s)$ be the optimal lower and upper bounds on $\alpha(s)$,  respectively. Then 
\begin{enumerate}[(i)]
    \item $\alpha_1(s) \ge \frac{2}{s+1}$ with equality at least for $s \le 2$.
    \item $\alpha_2(s) = 1$ for
    $s \le \gamma_1$, $\alpha_2(s) \le  \psi \frac{n}{n+1}$,
    for $s > \gamma_2$ 
    and $\alpha_2(s) \ge \frac{s}{s^2-1}$ for $s \le 2$.
\end{enumerate}

\item Let $\beta(s) \in \R$ such that $\left(\frac12 (C^\circ - C^\circ)) \right)^{\circ}\subset^{opt} \beta(s) \, \frac12 (C-C)$ and $\beta_1(s), \beta_2(s)$ be the optimal lower and upper bounds on $\beta(s)$, respectively. Then
\begin{enumerate}[(i)]
    \item 
    $\beta_1(s) \ge \frac{4s}{(s+1)^2}$ with equality at least for $s \le 2$.
   \item $\beta_2(s) = 1$ for $s \le \gamma_1$, $\beta_2(s) \le \mu \psi \frac{n(n+2)}{(n+1)^2}$ for $s > \gamma_3$ and
    $\beta_2(s) \ge \max \left\{ \frac{s}{s^2-1}, \frac{4s}{(s+1)^2} \right\}$ for $s \le 2$.
\end{enumerate}
\end{enumerate}
\end{thm}

Let us denote the \cemph{canonical basis} of $\R^n$ by $e^1,\dots,e^n\in\R^n$, the \cemph{Euclidean norm} of $x\in\R^n$ by $\|x\|$, and the \cemph{Euclidean unit ball} by $\B_2=\{x \in\R^n : \|x\|\leq 1\}$. For any $C,K \in \K^n$ the \cemph{Euclidean distance} is denoted by $d(C,K)$ and in case $C=\{p\}$ is a singleton, we abbreviate $d(\{p\},B)$ by $d(p,B)$. For any $C,K \in\K^n$ the \cemph{Banach-Mazur distance} between $K$ and $C$ is defined by $d_{BM}(K,C)=\inf\{\rho\geq 1:t^1 + K\subset L(C)\subset t^2+\rho K,\,L\in\mathrm{GL}(n), \, t^1,t^2\in\R^n\}.$ For every $X \subset \R^n$ let $\bd(X)$ and $\inter(X)$ denote the \cemph{boundary} and \cemph{interior} of $X$, respectively. 
For $C\in\K^n$ and $a\in\R^n$ let $\|x\|_C=\inf\{\rho>0:x\in\rho C\}$ be the \cemph{gauge function} of $C$ in $x$ and $h_C(a)=\sup\{a^Tx : x\in C\}$
be the \cemph{support function} of $C$ in $a$. Notice that $\|\cdot\|_C$ is a norm in the classic sense if and only if $C\in\K^n_0$ and remember that $\|x\|_{C}=h_{C^\circ}(x)$ for every $C\in\K^n$ and $x\in\R^n$ (see \cite{MR}).
For any $a\in\R^n \setminus \{0\}$
and $\rho\in\R$, $H^{\le}_{a,\rho} = \{x\in\R^n: a^Tx \leq \rho\}$ denotes the  \cemph{halfspace} with outer normal $a$ and right-hand side $\rho$.
We say that the halfspace $H^{\le}_{a,\rho}$ \cemph{supports} $C \in\K^n$ at $q \in C$, if $C \subset H^{\le}_{a,\rho}$ and $q \in \bd(H^{\le}_{a,\rho})$. For any $C\in\K^n$ and $p\in \bd(C)$, the \cemph{outer normal cone} of $K$ at $p$ is defined as
$N(C,p) = \{ a \in \R^n : a^Tp \geq a^Tx \text{ forall } x \in C\}$. 
For every $X\subset\R^n$ let us denote by $\pos(X)$, and $\aff(X)$ the
\cemph{positive} and \cemph{affine hull} of $X$, respectively, while 
the \cemph{relative interior} of $X$ is denoted by $ \relint (X)$. 
In case $u^1,\dots,u^{n+1}\in\R^n$ are affinely independent, we say that $\conv(\{u^1,\dots,u^{n+1}\})$ is an \cemph{$n$-simplex}.



\section{Preliminary results and lemmas}\label{sec:prelim_lemmas}

We recall the characterization of the optimal containment under homothety in terms of the touching conditions (see \cite[Theorem 2.3]{BrK}).
\begin{proposition}\label{prop:Opt_Containment}
Let $K,C\in\mathcal K^n$ and $K\subset C$. The following are equivalent:
\begin{enumerate}[(i)]
\item $K\subset^{opt}C$.
\item There exist $k\in\{2,\dots,n+1\}$, $p^j\in K\cap \bd(C)$, $u^j\in N(C,p^j)$, $j=1,\dots,k$, such that
$0\in\conv(\{u^1,\dots,u^k\})$.
\end{enumerate}
Moreover, if $K,C\in\mathcal K^n_0$, then (i) and (ii) are also equivalent to $K\cap\bd(C)\neq\emptyset$.
\end{proposition}

The next lemma shows that all the considered means are affine invariant. 
\begin{lemma}\label{lem:Means_Invariant}
	Let $K,C \in \K^n$ and $A$ be a non-singular affine transformation. Then
	\begin{equation*}
	\begin{split}
	A(K)\cap A(C)=A(K\cap C), \qquad \left( ((A(K))^\circ-(A(C))^\circ)/2 \right)^\circ=A\left((K^\circ- C^\circ)/2
	\right)^\circ, \\
	(A(K)+A(C))/2=A\left( (K+C)/2\right), \qquad \conv\left(A(K)\cup(A(C)\right)=A\left(\conv(K\cup C)\right).
	\end{split}
	\end{equation*}
\end{lemma}

\begin{proof}
From the fact that 	$A(C^\circ)=((A^{-1})^T(C))^\circ$ and since $A$ is non-singular, we get
\begin{equation*}
\begin{split}
\left(\frac{(A(K))^\circ-(A(C))^\circ}{2}\right)^\circ &=\left(\frac{(A^{-1})^T(K^\circ)-(A^{-1})^T(C^\circ)}{2}\right)^\circ \\
\left(\frac{(A^{-1})^T(K^\circ-C^\circ)}{2}\right)^\circ&=\left((A^{-1})^T \left( \frac{K^\circ-C^\circ}{2}\right) \right)^\circ= A \left( \left( \frac{K^\circ-C^\circ}{2}\right)^\circ  \right) .
\end{split}
\end{equation*}
The other identities are trivially true.
\end{proof}

The next result is a straightforward corollary of Lemma \ref{lem:Means_Invariant}.
\begin{cor}\label{cor:invariant}
Let $C \in \mathcal{K}^n$ be Minkowski centered, $A \in \mathbb R^{n \times n}$ a regular linear transformation and $\alpha \in \R$. Then
	\[
	C \cap (-C) \subset^{opt} \alpha \cdot \textrm{conv} ( C \cup (-C) )
	\] 	
	if and only if
	\[
	A(C) \cap A(-C) \subset^{opt} \alpha \cdot \textrm{conv} ( A(C) \cup A(-C)).
	\]
\end{cor}

The following proposition is an easy corollary out of Proposition \ref{prop:Opt_Containment}. It is a (variant of a) known result which in a more general version is given in \cite[(1.1)]{GrK} and we will use it to prove Lemma \ref{lem:Regard_of_Asym}.
\begin{proposition}\label{lem:polar}
	Let $ C, K \in \K^n_0$. 
	Then $C \subset^{opt} K$  if and only if $K^{\circ} \subset^{opt} C^{\circ}$. Moreover, the touching points of $C$ to the boundary of $K$ become the outer normals of supporting halfspaces to the touching points of $K^\circ$ to the boundary of $C^\circ$ and vice versa.
\end{proposition}

Let us mention that while the containment in Proposition \ref{lem:polar} holds for any $C,K$ with $0$ in their interior, the optimality of this containment may in general be lost even in case of Minkowski centered $C$ and $K$ (see Figure \ref{fig:polar-nonopt}).
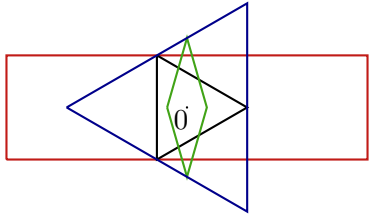
\begin{figure}[ht] 
\centering
    \begin{tikzpicture}[scale=0.8]
    \draw [thick, dred] (-3,-0.866) -- (-3,0.866) -- (3,0.866) -- (3,-0.866)-- (-3,-0.866) ;
    \draw [thick, black] (-0.5,0.866) -- (1,0) --  (-0.5,-0.866) -- (-0.5,0.866);
    \draw [thick, dgreen] (-0.33,0) -- (0,1.154) --  (0.33,0) -- (0,-1.154)-- (-0.33,0);
    \draw [thick, dblue] (-2,0) -- (1,1.732) --  (1,-1.732) -- (-2,0);
    \draw [fill] (0,0) circle [radius=0.01];
    \draw (-0.1,-0.2) node {$0$};
     \end{tikzpicture}
\caption{Minkowski centered $C$ (black) and $K$ (red), s.t.~$C \subset^{opt} K$ but $K^{\circ}$ (green) is not optimal contained in $C^{\circ}$ (blue). 
} 
\label{fig:polar-nonopt}
\end{figure}

As mentioned in the introduction, \eqref{eq:means_of_sets} is not left-to-right optimal for regular Minkowski centered simplices in even dimensions (while it is in odd dimensions). The following lemma prepares us to prove this fact in Lemma \ref{lem:Simplex_Odd}.

\begin{lemma}\label{lemma:opt_P_Ppolar}
	Let $P\in\K^n_0$ be a polytope and
	$v \in \bd(P)$ such that $v$ is also an outer normal of a closest facet of $P$ to the origin $0$, then
	\[
	P^\circ\subset^{opt}\frac{1}{\|v\|^2}P.
	\]
\end{lemma}

\begin{proof}
	Let $0<t_1\leq \cdots\leq t_m$ and $u^1,\dots,u^m\in \S^{n-1}$ be such that $P=\{x\in\R^n:|(u^i)^Tx|\leq t_i,\,i \in [m]\}.$ 
	Then $P^\circ=\conv(\{\pm u^1/t_1,\dots,u^m/t_m\})$,
	$t_1\B_2\subset^{opt}P$ and $t_1 u^1\in t_1\B_2\cap \bd(P)$.  
	Since  $\frac1{t_1} u^1\in P^\circ\cap \bd(\frac1{t_1}\B_2)$, we have $P^\circ\subset^{opt} \frac1{t_1} \B_2\subset^{opt} \frac1{t_1^2} P$
	and $\frac1{t_1} u^1$ is a common touching point of $P^\circ$ and $P$. Thus by part (iii) of Proposition \ref{prop:Opt_Containment}, we have $P^\circ\subset^{opt} \frac1{t_1^2}P$. Choosing $v = t_1u^1$  finishes the proof.
\end{proof}

We recall a stability result for the Banach-Mazur distance in the near-simplex case, given in \cite[Theorem 2.1]{Sch2}.    

\begin{proposition}\label{prop:Schneider}
Let $S \in \K^n$ be an $n$-simplex and $C\in\K^n$ such that $s(C)=n-\varepsilon$ and $\varepsilon\in(0,\frac 1 n)$. Then
\begin{equation}\label{eq:Schne_Stabil}
d_{BM}(C,S)\leq 1+\frac{(n+1)\varepsilon}{1-n\varepsilon}.
\end{equation}
\end{proposition}

\section{Optimality in Firey's inequality chain}
As we mentioned in the introduction, two of the containments in Proposition \ref{prop:means_of_sets} are always optimal in case of the symmetrizations.

\begin{lemma}\label{lem:Regard_of_Asym}
	Let $C \in \K^n$.
	Then
	\begin{enumerate}[(i)]
		\item $\frac{C-C}{2} \subset^{opt} \conv ( C \cup (-C) )$,
		\item $C \cap (-C) \subset^{opt} \left(\frac{C^\circ-C^\circ}{2}\right)^\circ$, if also $0 \in C$.
	\end{enumerate}
\end{lemma}

\begin{proof}
We start proving (i). By \eqref{eq:means_of_sets}, we have $\frac{1}{2}(C-C)  \subset  \conv(C \cup (-C))$.  
Now, there exists an extreme point $x$ of $C$ and an extreme point $y$ of $-C$, s.t.~$\conv(\{x,y\})$ is extreme in $\conv(C \cup (-C))$ and thus a halfspace $H^\le$ supporting $C$ and $-C$ in that edge.
This implies $\frac{1}{2}(x+y) \in (\frac{1}{2}(C-C))\cap \bd(\conv(C\cup(-C))$, and since $\frac{C-C}{2}$ and $\conv ( C \cup (-C) )$ are $0$-symmetric, Part
(iii) in Proposition \ref{prop:Opt_Containment} concludes the proof (i). 

Since $0 \in C$, we have $C \cap (-C) \neq \emptyset$. 
Applying (i) and Proposition \ref{lem:polar} implies (ii).
\end{proof}

Now we are ready for Lemma \ref{lem:Simplex_Odd}.

\begin{lemma}\label{lem:Simplex_Odd}
	Let $S$ be a Minkowski centered regular $n$-simplex. Then
	\begin{enumerate}[(i)]
		\item $S \cap (-S) \subset^{opt} \conv( S \cup (-S) )$, if $n$ is odd, 
		\item $S \cap (-S) \subset^{opt} \frac{n}{n+1} \conv( S \cup (-S) )$, if $n$ is even, and
		\item $\left(\frac{S^\circ-S^\circ}{2}\right)^\circ \subset^{opt}\frac{n(n+2)}{(n+1)^2} \cdot \frac{S-S}{2}$, if $n$ is even.
	\end{enumerate}
\end{lemma}

\begin{proof}
In order to simplify the calculations, we assume w.l.o.g.~that $S=\conv(\{p^1,\dots,p^{n+1}\})$ with $p^j \in \R^n$, such that $\|p^j \|=1$, $j \in [n+1]$. 
\begin{enumerate}[(i)]
\item Let $n \geq 1$ be odd and $p= \frac{2}{n+1} \left( p^1+ \cdots+  p^{\frac{n+1}{2}} \right)$. Since $\sum_{i=1}^{n+1} p^i=0$, we have $-p= \frac{2}{n+1} \left( p^{\frac{n+1}{2}+1}+ \cdots+  p^{n+1} \right) \in S$
and thus $p \in S \cap (-S)$. Define 
\begin{align*}  
	H_1^{\leq} &:= \left\{x \in  \R^n: (p^1+ \dots+  p^{\frac{n+1}{2}})^T x \leq \frac{n+1}{2n}  \right\},\\
	H_2^{\leq} &:=  \left\{x \in  \R^n: (-p^1 - \dots -  p^{\frac{n+1}{2}})^T x \leq \frac{n+1}{2n}  \right\}.
\end{align*}
Then we obtain for $j=1,\dots,(n+1)/2$
\[
|(p^1+\cdots+p^{\frac{n+1}{2}})^T p^j| = \left|1 - \left(\frac{n+1}{2}-1\right) \frac1n \right| = \frac{n+1}{2n},
\]
and for $j=(n+3)/2,\dots,n+1$ 
\[
|(p^1+\cdots+p^{\frac{n+1}{2}})^T p^j| = \left|\left(\frac{n+1}{2} (-\frac1n) \right)\right| = \frac{n+1}{2n},
\]
and therefore $S\subset H_1^{\leq}\cap H_2^{\leq}$.
Moreover, 
\[
|(p^1+ \dots+  p^{\frac{n+1}{2}})^T p | =  \frac{n+1}{2n},
\]
which shows that $H_1^{\le}$ and $H_2^{\le}$ support $S$ at $p$ and $-p$, respectively. Hence Part (iii) of Proposition \ref{prop:Charact_Opt_Means} is fulfilled, proving the optimal containment of $S \cap (-S)$ in $\conv(S \cup (-S))$.

\item We start observing that
 \[	
S\cap (-S) =\left\{x \in \R^n : | (p^j)^T x | \leq \frac{1}{n}, j \in [n+1] \right\}.
\]
Now, let $p=\lambda_1 p^1+\cdots+\lambda_{n+1}p^{n+1} \in S$ be a vertex of $S\cap(-S)$ for some $\lambda_j \geq 0$, $j \in [n+1]$ with $\sum_{i=1}^{n+1}\lambda_i=1$.
Since $p$ is a vertex, there must be $n$ of the constraints $| (p^j)^T x | \leq \frac{1}{n}, j \in [n+1]$ active in $p$. Hence, 
we may assume w.l.o.g.~that there exists $0\leq m\leq \frac n 2$, s.t.~$(p^j)^Tp = \frac1n$,  $j=1,\dots, \frac n2 + m$ and $(p^j)^Tp = -\frac1n$, $j=\frac n2+m+1,\dots,n$. Thus
\[
\frac{1}{n}= (p^j)^Tp = \lambda_j - \frac{1}{n}(\lambda_1+\cdots+\lambda_{j-1}+\lambda_{j+1}+\cdots+\lambda_{n+1}) = \lambda_j - \frac{1}{n}(1-\lambda_j)
\]
for $j=1,\dots,\frac n2 + m$ and
\[
-\frac{1}{n} = (p^j)^Tp = \lambda_j - \frac{1}{n}(\lambda_1+\cdots+\lambda_{j-1}+\lambda_{j+1}+\cdots+\lambda_{n+1})=\lambda_j-\frac{1}{n}(1-\lambda_j)
\]
for $j=\frac n2 + m+1,\dots,n$. This implies $\lambda_j = \frac 2 {n+1}$ for $j=1,\dots,\frac n2 + m$  and $\lambda_j=0$ for $j=\frac n2+m+1,\dots,n$.
Hence, $0 \le \lambda_{n+1} = 1-\sum_{i=1}^{n} \lambda_i = 1 - \frac{n+2m}{n+1} = \frac{1-2m}{n+1}$, which shows that $m=0$ and 
$p=\frac{2}{n+1}p^1+\cdots+\frac{2}{n+1}p^{\frac{n}{2}}+\frac{1}{n+1}p^{n+1}$.
We obtain
\[
\|p\|^2 = \frac{1}{(n+1)^2}\left(\left( \frac{n}{2} \cdot 4  + 1 \right) +    \left(\frac{n}{2} \cdot \left( \frac{n}{2} -1 \right) \cdot 4 + \frac{n}{2} \cdot 2+ \frac{n}{2} \cdot 2 \right)  \left( -\frac{1}{n} \right) \right)= \frac{1}{n+1} 
\]
and therefore $(n+1)p \in \bd((S\cap (-S))^\circ)$. Finally, since $S^\circ=-nS$, we have 
\[
(S\cap (-S))^\circ = \conv(S^\circ \cup (-S)^\circ) = n \cdot \conv(S \cup (-S)),\]
implying $\frac{n+1}{n} p  \in \bd(\conv( S \cup (-S) ))$ and therefore
\[	
S\cap (-S) \subset^{opt} \frac{n}{n+1} \conv( S \cup (-S) ).
\]
\item 
Notice that the faces of a Minkowski sum are Minkowski sums of their faces and 
\[
v := \frac12 \left(\sum_{i=1}^{\frac n 2} \frac{p^i}{\frac n 2} - \sum_{i= \frac n 2 + 1}^{n+1} \frac{p^i}{\frac n 2 + 1}\right) \in \bd\left( \frac{S-S}{2}\right) 
\]
is also the outer normal of one of the facets of $\frac {S-S} 2$, which are the closest to the origin.

We compute $\|v\|^2$
\[
\begin{split}
\|v\|^2 & = \frac{1}{4}\left(\frac{n}{2} \frac{1}{(\frac{n}{2})^2} + \left(\frac{n}{2} + 1\right)\frac{1}{\left(\frac{n}{2}+1\right)^2} + \frac{n}{2}\left(\frac{n}{2}-1\right) \frac{1}{(\frac{n}{2})^2} \left(-\frac{1}{n}\right) \right. \\
& \left. + \left(\frac{n}{2}+1\right)\frac{n}{2} \frac{1}{(\frac{n}{2}+1)^2}\left(-\frac{1}{n}\right)
- 2 \frac{n}{2}\left(\frac{n}{2}+1\right) \frac{1}{\frac{n}{2}} \frac{1}{\frac{n}{2}+1} \left(-\frac{1}{n}\right) \right) \\
& =\frac{1}{4}\left(\frac{2}{n}+\frac{2}{n+2}-\frac{n-2}{n^2} - \frac{1}{n+2} + \frac{2}{n}\right) = \frac{(n+1)^2}{n^2(n+2)},
\end{split}
\]
Using Lemma \ref{lemma:opt_P_Ppolar} and the identity $S^\circ=-nS$ again, we obtain 
\[
\left(\frac{S^\circ-S^\circ}{2}\right)^\circ=\frac1n\left(\frac{S-S}{2}\right)^\circ\subset^{opt}\frac{1}{n\|v\|^2}\frac{S-S}{2}=\frac{n(n+2)}{(n+1)^2}\frac{S-S}{2}.
\]
\end{enumerate}
\end{proof}

As mentioned in the introduction, Theorem \ref{thm:Charact_Opt_Means_KC} states that left-to-right optimality in \eqref{eq:means_of_sets} depends only on the optimal containment of the harmonic in the arithmetic mean.

\begin{proof}[Proof of Theorem \ref{thm:Charact_Opt_Means_KC}]

The forward direction directly follows from Proposition \ref{prop:means_of_sets}. Thus we only have to show the backward direction. 

Let $\left(\frac{K^\circ+C^\circ}2 \right)^\circ \subset^{opt} \frac{K+C}2$. By Proposition \ref{prop:Opt_Containment} there exist $k\in\{2,\dots,n+1\}$, $p^j\in \bd\left(\left(\frac{K^\circ+C^\circ}2 \right)^\circ \right) \cap \bd\left(\frac{K+C}2 \right)$, $u^j\in N(\frac{K+C}2,p^j)$, $j \in [k]$, such that $0\in\conv(\{u^1,\dots,u^k\})$. Choose any $p=p^j$, $u=u^j$ with $j \in [k]$, and $\beta \in \R$, such that $H_{u,\beta}$ is the hyperplane supporting $\frac12 (K+C)$ in $p$.
Since $p \in \bd\left(\left(\frac{K^\circ+C^\circ}2 \right)^\circ \right) \cap \bd\left(\frac{K+C}2 \right)$, we have 
\[
 \left\|p \right\|_{\left(\frac{K^\circ+C^\circ}2 \right)^\circ} = \left\|p \right\|_{\frac{K+C}2} =1.
\] 
Now,
on the one hand using $\frac{p}{\left\|p \right\|_{K}} \in K$ and $\frac{p}{\left\|p \right\|_{C}} \in C$, we see 
\[ 
\frac{1}{2} \left( \frac{1}{\left\|p \right\|_{K}} +  \frac{1}{\left\|p \right\|_{C}}  \right) p \in  \frac{K+C}{2},
\]
and therefore, 
\[ 
\left\|p \right\|_{\frac{K+C}2} \leq \left( \frac{1}{2} \left( \frac{1}{\left\|p \right\|_{K}} +  \frac{1}{\left\|p \right\|_{C}}  \right) \right)^{-1}.  
\]
On the other hand, since $h_{C^{\circ}} = \left\| \cdot \right\|_{C}$ (see \cite{Sch}), we have
\[ 
 \frac12 \left(\left\|p \right\|_{K}+\left\|p \right\|_{C}\right) = \frac12 \left(h_{K^{\circ}}(p)+ h_{C^{\circ}}(p)\right) = h_{\frac{K^\circ+C^\circ}2} = \left\|p \right\|_{\left(\frac{K^\circ+C^\circ}2\right)^\circ}. 
\]

Applying the arithmetic-harmonic mean inequality (for numbers - restating the main argument for Proposition \ref{prop:means_of_sets}) we obtain
\[
\left\|p \right\|_{\frac{K+C}2} 
\leq \left( \frac{1}{2} \left( \frac{1}{\left\|p \right\|_{K}} +  \frac{1}{\left\|p \right\|_{C}}  \right) \right)^{-1} 
\leq \frac12 \left(\left\|p \right\|_{K} + \left\|p \right\|_{C}\right)  
= \left\|p \right\|_{\left(\frac{K^\circ+C^\circ}2\right)^\circ}.  
\] 
However, since $p \in \bd \left(\left(\frac{K^\circ+C^\circ}2 \right)^\circ \right) \cap \bd \left(\frac{K+C}2 \right)$, it follows that \[
\left( \frac{1}{2} \left( \frac{1}{\left\|p \right\|_{K}} +  \frac{1}{\left\|p \right\|_{C}}  \right) \right)^{-1} 
= \frac12 \left(\left\|p \right\|_{K} + \left\|p \right\|_{C}\right).\] 
This means that we have equality between the harmonic and arithmetic mean of $\|p\|_K$ and $\|p\|_C$, which implies
\[ 
\left\|p \right\|_{K}= \left\|p \right\|_{C}= \left\|p \right\|_{\frac{K+C}2} =1 
\] 
and as a direct implication 
\[ 
\left\|p \right\|_{K \cap C} = \max \{\left\|p \right\|_{K}, \left\|p \right\|_{C}\}=1 . 
\] 
Now it suffices to show that $H_{u,\beta}$ also supports $\conv(K \cup C))$ at $p$. Assume that the latter is wrong. This would imply, that there exists $q \in K \setminus C$ or $q \in C \setminus K$ such that $u^T q > \beta$, say, w.l.o.g., $q \in K \setminus C$. However, this would imply $u^T \left(\frac{p+q}2\right) > \beta$, contradicting the fact that $H_{u,\beta}$ supports $\frac{K+C}2$. Hence $H_{u,\beta}$ supports also $\conv(K \cup C)$ at $p$. 

All together we see that $p^j\in (K \cap C) \cap \bd \left(\conv(K \cup C)) \right)$, with $u^j \in N(\conv(K \cup C),p^j)$, $j \in [k]$, and $0 \in \conv(\{u^1,\dots,u^k\})$. Using Proposition \ref{prop:Opt_Containment} we obtain the optimal containment of $K \cap C$ in $\conv(K \cup C)$.
\end{proof} 

The following proposition (see \cite[Theorem 1.3]{BDG}) is a direct application of Theorem \ref{thm:Charact_Opt_Means_KC} to $C$ and $-C$. 

\begin{proposition} \label{prop:Charact_Opt_Means}
	Let $C \in \K^n$ be such that $0 \in \inter(C)$.
	Then the following are equivalent:
	\begin{enumerate}[(i)]
		\item $C \cap (-C) \subset^{opt} \conv(C \cup (-C))$,
		\item $\left(\frac12 (C^\circ-C^\circ)) \right)^{\circ}\subset^{opt} \frac12 (C-C)$,
		\item there exist $p, -p \in \bd(C)$ and parallel halfspaces $H^{\le}_{a,\rho}$ and $H^{\le}_{-a,\rho}$ supporting $C$ at
      $p$ and $-p$, respectively.
	\end{enumerate}
\end{proposition}

In case the containment in \eqref{eq:means_of_sets} is left-to-right optimal for some Minkowski centered $C$ and $K=-C$, there also exist Minkowski centered bodies with an arbitrary smaller asymmetry providing left-to-right optimality in the full chain. 

\begin{lemma}\label{lem:Asym_Descent_Chain}
Let $C \in \K^n$ be Minkowski centered. If
	\[
	C \cap (-C) \subset^{opt} \conv( C \cup (-C) ),
	\] 	
	then for every $s\in[1,s(C)]$ there exists a Minkowski centered $C_s \in \K^n$ with $s(C_s)=s$, such that
	\[
	C_s \cap (-C_s) \subset^{opt} \conv( C_s \cup (-C_s) ).
	\] 	
\end{lemma}

\begin{proof}
Since $C$ is Minkowski centered, we obtain from Proposition \ref{prop:Opt_Containment} that there exist $p^1, \dots, p^k \in - \frac{1}{s(C)}C \cap \bd(C)$ with $k\in\{2,\dots,n+1\}$ and $u^j \in N(C,p^j)$, $j \in [k]$, s.t.~$0 \in \conv(\{u^1, \dots, u^k\})$. By Part (iii) of Proposition \ref{prop:Charact_Opt_Means} there also exist $p,-p\in (C\cap(-C)) \cap \bd(\conv(C\cup(-C)))$. For $t \in [0,1]$ let us define 
\[K_t:=\conv\left(\{p^1, \dots, p^k, \alpha_t p^1,  \dots,\alpha_t p^k,\pm p\}\right),\] with $\alpha_t:= -((1-t)s(C)+t)$. 
One may recognize, that since $\alpha_0 = -s(C)$ and $\alpha_0 p^j\in C$, $j \in [k]$, we have $K_t \subset C$. 
By the fact that $\pm p \in  K_t$, we have $\pm p \in (K_t \cap (-K_t)) \cap \bd(\conv(K_t\cup(-K_t))$ for all $t \in [0,1]$. Hence, $K_t$ fulfills Part (iii) of Proposition \ref{prop:Charact_Opt_Means} and therefore the optimal containment.
Moreover, $K_t \subset \alpha_t K_t$ with $\alpha_t p^j \in K_t \cap \bd(\alpha_t K_t)$ and $-u^j\in N(\alpha_t K_t,\alpha_t p^j)$, $j \in [k]$. Thus by Proposition \ref{prop:Opt_Containment} $K_t$ is optimally contained in $\alpha_t K_t$, which shows that $s(K_t)= - \alpha_t = (1-t)s(C) + t \in [1,s(C)]$. Choosing $C_s:=K_{\frac{s(C)-s}{s(C)-1}}$ concludes the lemma.
\end{proof}

Using Proposition \ref{prop:Schneider} we now prove Theorem \ref{thm:minMax_mean_improved}. 
\begin{proof}[Proof of Theorem \ref{thm:minMax_mean_improved}]  First of all let $\rho=d_{BM}(C,S)$, where $S$ is a regular Minkowski centered $n$-simplex and $\varepsilon := n - s$. Since $\gamma_3 > \gamma_2 > n - \frac{1}{n}$, we see that $C$ is under the conditions of Proposition \ref{prop:Schneider}. Hence, we may use \eqref{eq:Schne_Stabil} to obtain 
\begin{equation}\label{eq:rho}
\rho \leq 1+\frac{(n+1)\varepsilon}{1-n\varepsilon}=\rho_*.  
\end{equation}
Let $S=\conv(\{p^1,\dots,p^{n+1}\})$ be a regular Minkowski centered $n$-simplex, i.e., $\|p^i-p^j\|=const$ for all $i \neq j$, $i,j \in [n+1]$ with $\|p^i\|=n$. Moreover, let $F_i=\conv(\{p^j : j \neq i\})$ be the facet of $S$ opposing $p^i$ and  $L_i:=\aff(F_i)$ with $i \in [n+1]$. 
Applying a suitable regular linear transformation $L$, we may assume $c^1+S\subset L(C)\subset c^2+\rho S$ for some $c^1,c^2 \in\R^n$.
Since by Corollary \ref{cor:invariant} $L(C)\cap(-L(C))\subset  \gamma \cdot \conv(L(C)\cup(-L(C)))$ for some $\gamma>0$ is equivalent to $C\cap(-C)\subset \gamma \cdot\conv(C\cup(-C))$, we can replace w.l.o.g.~$C$ by $L(C)$ and assume 
\begin{equation}\label{eq:cont}
c^1+S \subset C \subset c^2+\rho S. 
\end{equation}

Let $\bar \mu \le 1$ be the minimal distance from $0$ to the facets of $c^1+S$, which is attained at $c^1+L_i$ for some $i\in[n+1]$. Since $C$ is Minkowski centered and $c^1+S \subset C$, we have $z:= c^1+p^i \in C$ and therefore $\frac {-z} {s(C)}  = \frac{-z}{n-\varepsilon} \in C$. 
Then  
\[
\begin{split}
d(z,\frac{p^i}n+L_i) + \bar \mu & = d(z,\frac {p^i}n + L_i) + d(\frac {p^i}n + L_i,c^1+L_i)\\
& =d(z,c^1+L_i) = n+1
\end{split}
\]
and we obtain
\begin{equation} \label{eq:xi1}
\xi:=d(\frac{-z}{n-\varepsilon},\frac{p^i}{n}+L_i)=\frac{d(z,\frac{p^i}n + L_i)}{n-\varepsilon} = \frac{n+1-\bar \mu}{n-\varepsilon}.
\end{equation}
Now we also have that $\frac{-z}{n-\varepsilon} \in C \subset c^2+\rho S$.
Since 
\[
d(c^1+L_i,c^2+\rho L_i)\leq (n+1)\rho-(n+1) = (n+1)(\rho-1),
\]
we see that 
\begin{equation} \label{eq:xi2}
\begin{split}
\xi & \leq d(c^2+\rho L_i, \frac {p^i}n + L_i)\\ 
&\leq d(c^2+\rho L_i,c^1+L_i)+d(c^1+L_i,\frac{p^i}n + L_i)\\
    & \leq (n+1)(\rho-1)+\bar \mu.
\end{split}
\end{equation}
Combining \eqref{eq:xi1} and \eqref{eq:xi2}, we obtain
\[
\frac{n+1-\bar \mu}{n-\varepsilon} \le (n+1)(\rho-1) +\bar \mu,
\]
which is equivalent to
\begin{equation}\label{eq:mu}
\bar \mu \geq \frac{n+1}{n+1-\varepsilon}(1-(n-\varepsilon)(\rho-1))=:\mu.
\end{equation}

Since $d(0,c^1+L_j) \geq \mu$ for every $j \in [n+1]$, this directly rewrites as
\begin{equation*}\label{eq:0_in_c^1_S}
0\in c^1+(1-\mu) S.
\end{equation*}
Now since
\[
0 \in c^1+(1-\mu) S \subset c^1+S \subset c^2+\rho S
\]
it holds $d(0,c^2+\rho L_j)\geq d(0,c^1+L_j)\geq \mu$ for every $j\in[n+1]$, which rewrites as
\begin{equation*}\label{eq:0_in_c^2_S}
0\in c^2+(\rho-\mu)S.
\end{equation*}
Moreover, using
\[
d(0,c^2+\rho L_j) \leq d(c^2+(\rho-\mu)p^j, c^2 + \rho L_j) = \rho + n(\rho-\mu)
\]
for every $j\in[n+1]$, thus 
\begin{equation}\label{eq:c_2_S_in_S}
c^2+\rho S \subset (\rho+n(\rho-\mu)) S.
\end{equation}
Moreover, since $d(0,c^1+L_j)\geq \mu$ for every $j\in[n+1]$, then 
\begin{equation}\label{eq:S_in_c^1_S}
\mu S \subset c^1+S.
\end{equation}
\begin{enumerate}[(i)]
\item 
Combining \eqref{eq:c_2_S_in_S} and \eqref{eq:S_in_c^1_S} with (ii) of Lemma \ref{lem:Simplex_Odd}
directly imply
\[
\begin{split}
C\cap(-C) & \subset (c^2+\rho S)\cap(-c^2-\rho S)\\
& \subset (\rho+n(\rho-\mu))(S \cap(-S)) \\
& \subset \frac{n}{n+1} (\rho+n(\rho-\mu))\mathrm{conv}(S\cup(-S))\\
& \subset \frac{n}{n+1} \frac{(\rho+n(\rho-\mu))}{\mu}\mathrm{conv}((c^1+S)\cup(-c^1-S))\\
& \subset\frac{n}{n+1}\frac{(\rho+n(\rho-\mu))}{\mu}\mathrm{conv}(C\cup(-C)).
\end{split}
\]
Since $\rho+n(\rho-\mu)$ is increasing in $\rho$, we obtain
\[
C\cap(-C)\subset\frac{n}{n+1}\frac{(\rho_*+n(\rho_*-\mu))}{\mu}\mathrm{conv}(C\cup(-C)). 
\]
Finally, from \eqref{eq:rho} and \eqref{eq:mu} (and remembering that $\varepsilon = n -s$) we obtain
\[
C \cap (-C) \subset \psi \frac{n}{n+1} \conv(C\cup(-C)),\] where $\psi =\frac{(n-s+1)(s+1)}{1-n(n-s)(n+s(n+1))} - n$.

Notice that solving $\psi \frac{n}{n+1} = 1$ becomes a quadratic in $s$ equation and the unique positive root has the expression
\[
\gamma_2 = \frac{n^4+n^3+2n^2+\sqrt{n^8+6n^7+17n^6+28n^5+28n^4+12n^3-4n^2-12n-4}}{2(n^3+2n^2+3n+1)}.
\]
Using Lemma \ref{lem:Asym_Descent_Chain}
we conclude that $\psi \frac{n}{n+1} < 1$, whenever $s(C) > \gamma_2$.
\item From \eqref{eq:cont}, \eqref{eq:c_2_S_in_S}, and \eqref{eq:S_in_c^1_S} we obtain  
\[\mu S\subset c^1+S \subset C \subset c^2+\rho S  \subset(\rho+n(\rho-\mu))S.\]
Thus, using (iii) of Lemma \ref{lem:Simplex_Odd} this shows
\[
\begin{split}
\left(\frac{C^\circ-C^\circ}{2}\right)^\circ & \subset\left(\frac{(c_2+\rho S)^\circ-(c_2+\rho S)^\circ}{2}\right)^\circ
 \subset(\rho+n(\rho-\mu))\left(\frac{S^\circ-S^\circ}{2}\right)^\circ\\
& \subset \frac{n(n+2)}{(n+1)^2} (\rho+n(\rho-\mu))\frac{S-S}{2}\\
& =\frac{n(n+2)}{(n+1)^2} (\rho+n(\rho-\mu))\frac{(c_1+S)-(c_1+S)}{2}\\
& \subset \frac{n(n+2)}{(n+1)^2}(\rho+n(\rho-\mu))\frac{C-C}{2}
\end{split}
\]
and since $\rho+n(\rho-\mu)$ is increasing in $\rho$, we obtain
\[
\left(\frac{C^\circ-C^\circ}{2}\right)^\circ \subset \frac{n(n+2)}{(n+1)^2} (\rho_*+n(\rho_*-\mu))\frac{C-C}{2}.
\]
Finally, combining \eqref{eq:rho} and \eqref{eq:mu} and remembering that  $\varepsilon=n-s$, we obtain
\[
\left(\frac{C^\circ-C^\circ}{2}\right)^\circ \subset \mu \psi \frac{n(n+2)}{(n+1)^2}  \frac{C-C}{2}.
\]
Notice that $\mu \psi \frac{n(n+2)}{(n+1)^2} = 1$ also becomes a quadratic equation in $s$, which has a unique positive root of expression
\[
\gamma_3 = \frac{n^4+3n^3+2n^2+1+\sqrt{n^8+6n^7+13n^6+8n^5-14n^4-22n^3+8n+1}}{2(n^3+2n^2+2n)}.
\]
By Lemma \ref{lem:Asym_Descent_Chain}
we conclude that $\mu \psi \frac{n(n+2)}{(n+1)^2} < 1$, whenever 
\[
s(C)>\frac{n^4+3n^3+2n^2+1+\sqrt{n^8+6n^7+13n^6+8n^5-14n^4-22n^3+8n+1}}{2(n^3+2n^2+2n)},
\]
which itself is greater than $n- \frac1n$.
\end{enumerate}
\end{proof}

We now provide the proof of Theorem \ref{thm:gamma}. 
\begin{proof}[Proof of Theorem \ref{thm:gamma}] 
From Theorem \ref{thm:minMax_mean_improved} we directly obtain 
$\gamma(n) \le \gamma_2$ for even $n$.

In order to obtain a lower bound on $\gamma(n)$ in even dimensions, we provide a suitable family of sets with left-to-right optimal containment in \eqref{eq:means_of_sets} and asymmetry $\gamma_1 > n-1$ by extending the construction of the Golden House in \cite{BDG}. Note that this construction also holds in odd dimensions.

Let $S$ be a Minkowski centered regular $n$-simplex such that without loss of generality $S=\conv(\{p^1,\dots,p^{n+1}\})$ with $p^j \in \R^n$ and $\|p^j \|=1$, $j \in [n+1]$. Note that $(p^i)^Tp^j=-1/n$ for $i \not =j$. Now we define $C=\conv(\{p^1,\dots,p^{n+1}\})\cap H^{\pm}$ with $H^{\pm} := \{ x \in \R^n : \pm (p^1-p^2)^Tx \leq \eta \}$, where $\eta=(p^1-p^2)^T p \in (0,1+\frac1n)$ and $p=(1-\lambda)p^1+\lambda u^2 \in \bd(H^+)$ for some $\lambda \in[0,1]$. 
Then $\eta = 1 - \lambda - \lambda + \frac{1 - \lambda}{n} - \frac{\lambda}{n}$ and therefore
\begin{equation}\label{lemma3.6:lambda}
\lambda=\frac{1+\frac1n-\eta}{2(1+\frac1n)}.
\end{equation}
Because of the symmetry of $C$ with respect to the axis orthogonal to $\aff(p^1,p^2)$, there exist $\nu \in \R$ and $s \in [1,n]$ such that 
\begin{equation*} 
\nu(p^1+p^2)-\frac1sC \subset^{opt} C, 
\end{equation*}
which can be rewritten as
\begin{equation}\label{lemma3.6:c}
-\frac1s(C-c) \subset^{opt} C-c \quad \text{for} \quad c=\frac{s}{s+1}\nu(p^1+p^2),
\end{equation}
such that $c$ is the Minkowski center of $C$.

\begin{figure}[H]
\centering
    \begin{tikzpicture}[scale=1.5]
    \draw [thick] (0,1.61) -- (1,0)-- (1,-1)-- (-1,-1)-- (-1,0)--(0,1.61);
    \draw [thick, dashed, rotate around={180:(0,0)}, scale=0.615] (0,1.61) -- (1,0)-- (1,-1)-- (-1,-1)-- (-1,0)--(0,1.61);;
    \draw [thick, gray] (1,-1.2)-- (1,0.2);
    \draw [thick, gray] (-1,-1.2)-- (-1,0.2);
    \draw [thick, dashed, gray] (1,0)-- (1.6,-1);
    \draw [thick, dashed,gray] (1.6,-1)-- (1,-1);
    \draw [thick, dashed, gray] (-1,0)-- (-1.6,-1);
    \draw [thick, dashed,gray] (-1.6,-1)-- (-1,-1);

    \draw [fill] (0,0) circle [radius=0.01];
    \draw [fill] (1,0) circle [radius=0.02];
    \draw [fill] (1,-1) circle [radius=0.02];
    \draw [fill] (-1,-1) circle [radius=0.02];
    \draw [fill] (-1,0) circle [radius=0.02];
    \draw [fill] (0,1.61) circle [radius=0.02];
    \draw [fill] (0.61,0) circle [radius=0.02];
    \draw [fill] (-0.61,0) circle [radius=0.02];
    \draw [fill] (0,-1) circle [radius=0.02];
    \draw [fill] (0.61,0.61) circle [radius=0.02];
    \draw [fill] (-0.61,0.61) circle [radius=0.02];
    \draw [fill] (-1.6,-1) circle [radius=0.02];
    \draw [fill] (1.6,-1) circle [radius=0.02];
    
    \draw (-0.1,-0.1) node {$c$};
    
    \draw (1.4,0.4) node {$\bd(H^+)$};
    \draw (-1.4,0.4) node {$\bd(H^-)$};
    \draw (0,1.85) node {$p^3$};
    \draw (-1.8,-1.1) node {$p^2$};
    \draw (1.8,-1.1) node {$p^1$};
    \draw (-1.5,0.8) node {$\nu (p^1+p^2)- \frac{p}{s}$};
    \draw (1.2,-1.2) node {$p$};
    \draw (1.2,-0.02) node {$q$};
    
     \end{tikzpicture}
\caption{Construction from the proof of Theorem \ref{thm:gamma} for $n=2$: \\ $C$ (black), $-\frac1s C$ (dashed) and $\bd(H^+)$, $\bd(H^-)$ (gray dashed). 
}
\end{figure}
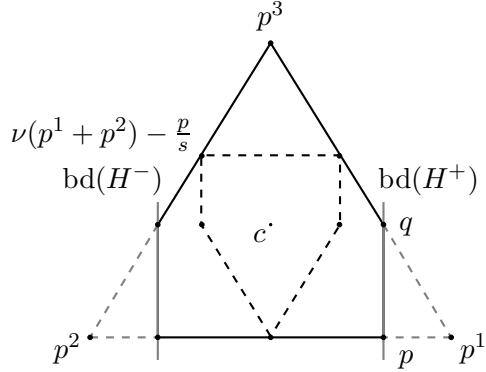

Since $\nu(p^1+p^2)-\frac{p}{s}$ and $\nu(p^1+p^2)-\frac{p^3}{s}$ belong to the facets of $S$ with outer normals $-p^1$ and $-p^3$, respectively, we obtain
\[
\left( \nu(p^1+p^2)-\frac{p}{s}\right)^T(-p^1) =\frac1n
\quad\text{ and } \quad
\left( \nu(p^1+p^2)-\frac{p^3}{s} \right)^T (-p^3) =\frac1n.
\]
The latter two conditions translate into
\[
-\left(\nu-\frac{1-\lambda}{s}\right)+\left(\nu-\frac{\lambda}{s}\right)\frac1n=\frac1n 
\quad\text{ and } \quad
\frac{2\nu}{n}+\frac1s = \frac1n,
\]
which can be simplified to
\begin{equation}\label{lemma3.6:mu}
s=n-2\lambda \quad \text{ and } \quad \nu=\frac{\lambda}{2\lambda-n}.
\end{equation}
Inserting \eqref{lemma3.6:lambda} we obtain
\begin{equation}\label{lemma3.6:mu1a}
s=\frac{n- \frac1n+\eta}{1+\frac1n} \quad \text{ and } \quad
\nu=\frac{1+\frac1n - \eta}{2 \left( \frac1n -n- \eta \right)}.
\end{equation}
Next, let $q=c+\xi(p^1-p^2)$, with $\xi>0$ such that $q \in \bd(H^+)$, which belongs to the facet of $S$ with outer normal $-p^2$. 
Using $(p^1+p^2)^T(p^1-p^2)=0$, $(p^1-p^2)^T(p^1-p^2)=2 \left( 1+ \frac1n \right)$ and \eqref{lemma3.6:c}, this implies
\begin{equation} \label{eq:insert-nu}
\left(\frac{s}{s+1}\nu+\xi\right)\frac1n-\left(\frac{s}{s+1}\nu-\xi\right) = \frac1n \quad \text{and} \quad  \eta = 2 \xi \left(1+\frac1n\right).
\end{equation}
Inserting $\eta$ from \eqref{eq:insert-nu} into \eqref{lemma3.6:mu1a} leads to
\[
s = n - 1 + 2\xi \quad \text{ and } \quad
\nu=\frac{1+\frac1n - 2\xi\left(1+\frac1n\right)}{2 \left( \frac1n -n- 2\xi \left(1+\frac1n\right)\right)}=\frac{1-2\xi }{1-2\xi-n}.
\] 
and inserting this result for $\nu$ in \eqref{eq:insert-nu} gives us
\begin{equation*}
\left(\frac{s}{s+1}\frac{1-2\xi }{1-2\xi-n}+\xi\right)\frac1n-\left(\frac{s}{s+1} \frac{1-2\xi }{1-2\xi-n}-\xi\right) = \frac1n,
\end{equation*}
which one can solve for $\xi$ and with it for $s$ to obtain
\[
\xi=\frac{1-n+\sqrt{(n-2)n+5}}{4} \quad \text{ and } \quad s=n+2\xi-1=\frac{n-1+\sqrt{(n-2)n+5}}{2} = \gamma_1.
\]

Since condition (iii) of Proposition \ref{prop:Charact_Opt_Means} is fulfilled
for the Minkowski centered $C-c$ at the points $\pm \xi(p^1-p^2)$, \eqref{eq:means_of_sets} is optimal for $C-c$ and $c-C$,
and $s(C-c)=\gamma_1$, as desired.

Finally, by Lemma \ref{lem:Asym_Descent_Chain} we see that for
$s \le \gamma_1$ there exists a Minkowski centered $C\in\K^n$ such that $C\cap(-C)\subset^{opt}\conv(C\cup(-C))$, proving
$\gamma(n) \ge \gamma_1$.
\end{proof}

\section{Reverse containment}\label{sec:reverse_inclusions}
In this section we prove Theorem \ref{thm:reverse_inclusions}. While the proof of Parts (i)-(v) is straightforward, understanding (vi) needs some additional effort: on the one hand, we show that $C=S\cap(-sS)$, where $S$ is a Minkowski regular simplex, provides optimality in (vi) for each $s\in[1,n]$, while on the other hand, we find a more intriguing family of sets not fulfilling that optimality (see Example \ref{ex:notopt}).

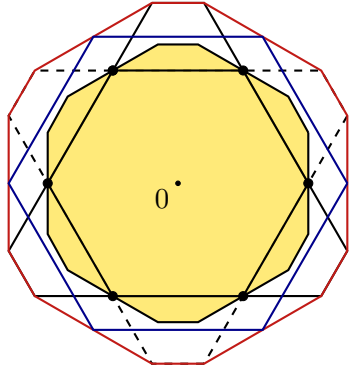
\begin{figure}[H]
\centering
    \begin{tikzpicture}[scale=3]
    \draw [thick, black] (-0.11547,0.8) -- (0.11547,0.8) -- (0.75038,-0.2997)-- (0.6348,-0.5)-- (-0.6348,-0.5)-- (-0.75038,-0.2997)--(-0.11547,0.8);
    \draw [thick, dashed] (0.11547,-0.8) -- (-0.11547,-0.8) -- (-0.75038,0.2997)-- (-0.6348,0.5)-- (0.6348,0.5)-- (0.75038,0.2997)--(0.11547,-0.8);
    \draw [thick, dred] (0.75038,-0.2997) -- (0.6348,-0.5)-- (0.11547,-0.8) -- (-0.11547,-0.8)--(-0.6348,-0.5)-- (-0.75038,-0.2997)--(-0.75038,0.2997)-- (-0.6348,0.5)--(-0.11547,0.8) -- (0.11547,0.8)--(0.6348,0.5)-- (0.75038,0.2997)--(0.75038,-0.2997);
    \draw [thick, dblue] (0.75038,0) -- (0.375135,-0.65) -- (-0.375135,-0.65)-- (-0.75038,0) -- (-0.375135,0.65)--  (0.375135,0.65)-- (0.75038,0);
    \draw [thick, fill=lgold, fill opacity=0.7] (0.5774,0.225)--(0.5774,-0.225) --(0.4887,-0.3849) -- (0.089,-0.6158) -- (-0.089,-0.6158) -- (-0.4887,-0.3849) -- (-0.5774,-0.225) --(-0.5774,0.225) --(-0.4887,0.3849) --(-0.089,0.6158) --(0.089,0.6158) --(0.4887,0.3849)--(0.5774,0.225);
    \draw [thick, black] (0.5774,0) -- (0.28885,-0.50035) -- (-0.28885,-0.50035) -- (-0.5774,0) -- (-0.28885,0.50035) -- (0.28885,0.50035) -- (0.5774,0);
    \draw [fill] (0,0) circle [radius=0.01];
    \draw [fill] (0.5774,0) circle [radius=0.02];
    \draw [fill] (-0.28885,0.50035) circle [radius=0.02];
    \draw [fill] (-0.28885,-0.50035) circle [radius=0.02];
    \draw [fill] (-0.5774,0) circle [radius=0.02];
    \draw [fill] (-0.28885,0.50035) circle [radius=0.02];
    \draw [fill] (0.28885,0.50035) circle [radius=0.02];
     \draw [fill] (0.28885,-0.50035) circle [radius=0.02];
    \draw (-0.07,-0.07) node {$0$};
     \end{tikzpicture}
\caption{$C=S \cap (-s S)$ (where $S$ is a regular triangle and $s=s(C)=1.5$, c.f.~Remark \ref{ex:partial-symm-of-simplex}) (black), $-C$ (dashed), $C \cap (-C)$ (convex hull of black points), $\conv (C \cup (-C))$ (red), $\frac{C-C}{2}$ (blue), and  $\left(\frac{C^\circ+(-C)^\circ}{2}\right)^{\circ}$ (yellow). 
}
\end{figure}

\begin{remark}\label{ex:partial-symm-of-simplex}
Let $C=S\cap(-sS)$, where $S=\conv(\{p^1,\dots,p^{n+1}\})$ with $\|p^i\|=1$, $i \in [n+1]$, is a regular simplex centered at $0$, and $s\in[1,n]$. Notice that
$C=S\cap(-sS)\subset s^2S\cap(-sS)=-sC$.
Let $G_i=\{x\in S:(p^i)^T x = -\frac 1 n\}$ and $F_i=S\cap(-sG_i)$ with $i\in[n+1]$, then 
$G_i \subset \bd(C)$ and therefore $F_i\subset \bd(-sC)$. 
Moreover, the points $\frac s n p^i$ belong to $C \cap \bd(-sC)$, $i \in [n+1]$, with $p^i$ being a normal vector of $-sC$ in $\frac s n p^i$. Thus by Proposition \ref{prop:Opt_Containment} we conclude that $C \subset^{opt}-sC$ and therefore that $C$ is Minkowski centered and $s(C)=s$.
\end{remark}

\begin{proof}[Proof of Theorem \ref{thm:reverse_inclusions}]
We begin the proof by showing 
containments in Parts (i) - (v).

For Part (ii) notice that $-C \subset sC$ directly implies $(s+1)(-C) \subset s(C-C)$, and therefore $-C \subset \frac{2s}{s+1} \frac{C-C}{2}$. Using the $0$-symmetry of $\frac {C-C}2$, we obtain $\conv(C\cup(-C)) \subset \frac{2s}{s+1} \frac{C-C}{2}$.

Part (iii) now follows directly from Part (ii) by using Proposition \ref{lem:polar}.

Since $-C\subset sC$, we have $\frac{C-C}{2} \subset \frac{s+1}{2}C$. From the $0$-symmetry of $\frac{C-C}{2}$, we obtain $ \frac{C-C}{2} \subset \frac{s+1}{2} (C\cap(-C))$, which yields Part (iv).

Part (v) then follows from Part (iv) by using Proposition \ref{lem:polar} again.

Finally, notice that Parts (ii) and (iv) together directly yield Part (i). 

We now show that the containment in (i) is optimal for every $C$. Indeed, 
let $p\in -C\cap \bd(sC)$. Let $H$ be a hyperplane supporting $sC$ at $p$, then $H'$, which is the translation of $H$ with
$p/s\in H'$ supports $-C$ at $p/s$, from which $p\in\bd(\conv(C\cup(-C)))$. Moreover, since $p\in s C\cap(-C)$ and $p\in\bd(sC)$, then $p\in \bd(s(C\cap(-C)))$, and thus Proposition \ref{prop:Opt_Containment} yields the result.

Since Part (i) shows optimality for every $C$, and since either joining Parts (ii),(iv) or (iii),(v) recovers (i), each of the Parts (ii)-(v) must be optimal for every $C$. 

The proof of Part (vi) is a bit more subtle. 
Let $\frac{C-C}{2} \subset^{opt} \alpha \left(\frac{C^\circ-C^\circ}{2}\right)^\circ$ for some $\alpha \ge 1$. 
Using the polarity of gauge and support function and the linearity of the latter with respect to Minkoswki addition, it holds
$\|z \|_{\left(\frac{C^\circ-C^\circ}{2}\right)^\circ} = h_{\frac{C^\circ - C^\circ}{2}}(z) = \frac12 \left(h_{C^\circ}(z) + h_{-C^\circ}(z)\right) = \frac12 (\|z \|_C+ \|z \|_{-C})$. From this and Part (iv) we obtain 
\begin{align} \label{eq:AMinHMupper}
  \alpha &= \max_{z \in \frac{C-C}{2}}
\|z \|_{\left(\frac{C^\circ-C^\circ}{2}\right)^\circ} =  \max_{z \in \frac{C-C}{2}} \frac12 \left( \|z \|_C+ \|z \|_{-C} \right) \\ & \leq  \max_{z \in \frac{C-C}{2}} \max\left\{\|z \|_C,\|z \|_{-C}\right\} = \max_{z \in \frac{C-C}{2}} \|z\|_{C \cap (-C)} = \frac{s+1}{2}. \nonumber
\end{align}
Now let $C=S\cap(-sS)$ for $s\in[1,n]$ as given in Remark \ref{ex:partial-symm-of-simplex} and let $q^{n+1}$, $q^n$
be the centers of the $(n-2)$-dimensional facets
of $F_{n+1}$ and $F_{n}$, respectively, which do not contain a vertex belonging to the line segment $[p^n,p^{n+1}]$.
Now let $v$ be a vertex of $F_{n+1}$ with the outer normal $p^{n+1}$. Then $v$ belongs to an edge connecting $p^{n+1}$ with 
$\conv(\{p^1,\dots,p^{n}\})$. Thus, $v=(1-\lambda)p^i+\lambda p^{n+1}$ for some $\lambda\in[0,1]$ and $i \neq n+1$. 
Since $v \in F_{n+1}$, we know by Remark \ref{ex:partial-symm-of-simplex} that
\[
\frac{s}{n}=((1-\lambda)p^i+\lambda p^{n+1})^Tp^{n+1}=(1-\lambda)\frac{-1}{n}+\lambda,
\]
i.e.~$\lambda=\frac{s+1}{n+1}$. Thus, the vertices of $F_{n+1}$ are $\frac{n-s}{n+1}p^i+\frac{s+1}{n+1}p^{n+1}$, $i \in [n]$ and
therefore 
\[
\begin{split}
q^{n+1} & =\frac{1}{n-1}\sum_{i=1}^{n-1}\left(\frac{n-s}{n+1}p^i+\frac{s+1}{n+1}p^{n+1}\right) \\
& = \frac{1}{n-1}\left(\frac{(n-1)(s+1)}{n+1}p^{n+1}+\frac{n-s}{n+1}(-p^n-p^{n+1})\right) \\
& = \frac{1}{n^2-1}((ns-1)p^{n+1}-(n-s)p^n),
\end{split}
\]
where we have used that $\sum_{i=1}^{n-1}p^i=-p^n-p^{n+1}$.
For the same reasons 
\[
-q^n=\frac{1}{n^2-1}((n-s)p^{n+1}-(ns-1)p^n).
\]
Now, let $z := \frac{q^{n+1}-q^n}{2} \in\frac{C-C}{2}$. Then
\[
\begin{split}
z &= \frac{1}{2(n^2-1)}\left((ns+n-s-1)p^{n+1}-(ns+n-s-1)p^n\right) =\frac{s+1}{2(n+1)}(p^{n+1}-p^n).
\end{split}
\]
Moreover, one should notice that when $s=1$ 
\[
q^{n+1} = -q^{n} = \frac{1}{n+1}(p^{n+1}-p^n) \in \bd\left(\frac{(S \cap (-S)) - (S \cap (-S))}{2}\right)=\bd(S \cap (-S)).
\]

Since we also have $S \cap (-S) = C \cap (-C)$, independently of our choice of $s\in[1,n]$, it follows
\[
\|z\|_{C}=\|z\|_{-C}=\|z\|_{C\cap(-C)}=\|z\|_{S \cap (-S)}=\frac{s+1}{2}\left\|\frac{1}{n+1}(p^{n+1}-p^n)\right\|_{S \cap (-S)}=\frac{s+1}{2},
\]
which shows equality in \eqref{eq:AMinHMupper} and thus that $\frac{C-C}{2} \subset^{opt} \frac{s+1}{2} \left(\frac{C^\circ-C^\circ}{2}\right)^\circ$, concluding the proof.
\end{proof}

Note that only in Part (vi) of Theorem \ref{thm:reverse_inclusions} the containment may not always be optimal. Below we give two examples: the first one shows that at least in 2-space all regular $k$-gons achieve optimality, while the second one provides a construction of sets in arbitrary dimensions where the containment is not optimal.

\begin{example}
Let $C\subset\mathbb R^2$ be a Minkowski centered regular $k$-gon with odd $k$. 
Then 
\[
\frac{C-C}{2}\subset^{opt} \frac{s(C)+1 }{2}\left(\frac{C^\circ-C^\circ}{2}\right)^\circ.
\] 
\end{example}
\begin{proof}
Let $r(C)$, $R(C)$ be the in- and circumradius of $C$ in the Euclidean distance, respectively. Assume w.l.o.g.\ that $R(C)=1$. 
Since the Minkowski center of $C$ coincides with the in- and circumcenter, we can easily conclude that $s(C)=\frac{R(C)}{r(C)}$.
Now, since for any $k$-gon $r(C)=R(C) \cos \left(\frac{ \pi}{k} \right)$, it follows 
\[
s:=s(C)=\frac{R(C)}{R(C) \cos \left(\frac{ \pi}{k} \right)}= \frac{1}{\cos \left(\frac{ \pi}{k} \right)}. 
\] 
We choose a vertex $p= \frac{u-v}{2}$ of $\frac{C-C}{2}$ , where $u$ and $v$ are vertices of $C$. Assume w.l.o.g.\ that $\|u\|=\|v\|=1$. Then since $C$ is Minkowski centered, $C^\circ =\rho (-C)$ for some $\rho >0$. Note that $R(C)=1$, thus $r(C^\circ)=1$ and $\rho=s$. 
Since $C^\circ$ is again a regular $k$-gon,  the distance from the origin to any edge
of $\left(\frac{C^\circ-C^\circ}{2}  \right)^{\circ}$ is the same.

Using $C^\circ =s (-C)$, we have that $w=s p$ is a vertex of $\frac{C^\circ-C^\circ}{2}$. Moreover, $\{ x \in \R^n :  w^T x =\|w\|^2 \}$ determines an edge of $\left(\frac{C^\circ-C^\circ}{2}\right)^\circ$ with outer-normal vector $\frac{w}{\|w \|^2}$. This implies $\frac{C^\circ-C^\circ}{2}  \subset^{opt} \| w \|^2 \left(\frac{C^\circ-C^\circ}{2}\right)^\circ.$ 
Since $\cos \left(\frac{ \pi}{k} \right) =\frac{1}{s}$ and $R(C)=1$, we get
\begin{align*}
   \| w \|^2= 
   \frac{ s^2}{4} \| u-v \|^2 &= \frac{ s^2}{4} \left( \| u\|^2+\|-v \|^2 +2 u^T(-v) \right) = \frac{ s^2}{4} \left( 2 R(C)^2+ 2R(C)^2 \cos \left(\frac{ \pi}{k} \right) \right) \\
    &=\frac{ s^2}{2}  \left( 1+ \cos \left(\frac{ \pi}{k} \right) \right) 
    = \frac{ s^2}{2}  \left( 1+ \frac{ 1}{s} \right).
\end{align*}
Therefore,
\[
\frac{C-C}{2} = \frac{1}{s} \frac{C^\circ-C^\circ}{2}  \subset^{opt}  \frac{1}{s} \frac{ s^2}{2}  \left( 1+ \frac{ 1}{s} \right) \left(\frac{C^\circ-C^\circ}{2}\right)^\circ=\frac{s+1}{2} \left(\frac{C-C}{2}\right)^\circ . 
\]
\end{proof}

\begin{example}\label{ex:notopt} 
Note that here we provide a construction only for $n=2$. For the higher dimensional case one can simply embed the construction below into the according space, keeping the Minkowski center 0. This keeps its asymmetry value and the same factor for the containment of the arithmetic mean within the harmonic mean. Thus, the construction essentially provides a family of sets in arbitrary dimensions with asymmetry $s  \in (1,2)$ (only), such that the arithmetic mean is contained in the interior of the harmonic mean scaled by $\frac{s+1}2$.

Let $K=S\cap(-sS)$, where $S$ is a Minkowski centered regular triangle and $s\in(1,2)$. By  $p^1,\dots,p^6$ we denote the vertices of $K$, counted in clockwise order, such that $[p^i,p^{i+1}]$ with $i=1,3,5$ are the shorter edges of $K$. 
Let
\[
C=\conv(\{p^2, p^4, p^6, \frac{p^1+p^2}{2},  \frac{p^3+p^4}{2}, \frac{p^5+p^6}{2}, \frac{p^1-p^4}{s+1}, \frac{p^3-p^6}{s+1}, \frac{p^5-p^2}{s+1}\})
\]
(c.f.~Figure \ref{fig:ScapsS_cut}).
\begin{figure}[H]
\centering
    \begin{tikzpicture}[scale=2.7]
    \draw [thick, black] (0,0.8) -- (0.11547,0.8)-- (0.28885,0.50035)-- (0.5774,0)-- (0.69,-0.4)--(0.6348,-0.5)--(0.28885,-0.50035)--(-0.28885,-0.50035)-- (-0.69,-0.4)--(-0.75038,-0.2997)--(-0.5774,0)--(-0.28885,0.50035)--(0,0.8);
    \draw [thick, dashed] (0,-0.8)--(-0.11547,-0.8) --(-0.28885,-0.50035)--(-0.5774,0)--(-0.69,0.4)--(-0.6348,0.5)--(-0.28885,0.50035)--(0.28885,0.50035)--(0.69,0.4)-- (0.75038,0.2997)--(0.5774,0)--(0.28885,-0.50035)--(0,-0.8);
     \draw [thick, black, dotted] (-0.11547,0.8) -- (0.11547,0.8) -- (0.75038,-0.2997)-- (0.6348,-0.5)-- (-0.6348,-0.5)-- (-0.75038,-0.2997)--(-0.11547,0.8);
    \draw [thick, black, dotted] (0.11547,-0.8) -- (-0.11547,-0.8) -- (-0.75038,0.2997)-- (-0.6348,0.5)-- (0.6348,0.5)-- (0.75038,0.2997)--(0.11547,-0.8);
    \draw [thick, dred, dotted] (0.75038,-0.2997) -- (0.6348,-0.5)-- (0.11547,-0.8) -- (-0.11547,-0.8)--(-0.6348,-0.5)-- (-0.75038,-0.2997)--(-0.75038,0.2997)-- (-0.6348,0.5)--(-0.11547,0.8) -- (0.11547,0.8)--(0.6348,0.5)-- (0.75038,0.2997)--(0.75038,-0.2997);
    \draw [thick, dred] (0.69,-0.4)--(0.6348,-0.5)-- (0,-0.8) -- (-0.11547,-0.8) --(-0.69,-0.4)--(-0.75038,-0.2997)--(-0.69,0.4)--(-0.6348,0.5)--(0,0.8)--(0.11547,0.8)--(0.69,0.4)-- (0.75038,0.2997)--(0.69,-0.4);
    \draw [thick, dblue, dashed] (0.75038,0) -- (0.375135,-0.65) -- (-0.375135,-0.65)-- (-0.75038,0) -- (-0.375135,0.65)--  (0.375135,0.65)-- (0.75038,0);
    \draw [thick, dblue] (0.46,-0.50035) --(0.3,-0.65)--(-0.2,-0.65)--(-0.4,-0.6) --(-0.46,-0.50035)--(-0.66,-0.15)--(-0.72,0.05) --(-0.46,0.50035)-- (-0.3,0.65)--(0.2,0.65)-- (0.4,0.6)--(0.665,0.15)-- (0.72,-0.05)--(0.46,-0.50035);
     \draw [thick, fill=lgold, fill opacity=0.7] (-0.28885,0.50035)--(0.28885,0.50035)--(0.5774,0)--(0.28885,-0.50035)--(-0.28885,-0.50035)--(-0.5774,0);
    \draw [thick, black] (0.5774,0) -- (0.28885,-0.50035) -- (-0.28885,-0.50035) -- (-0.5774,0) -- (-0.28885,0.50035) -- (0.28885,0.50035) -- (0.5774,0);
    \draw [thick, black] (0,0) -- (0.4,0.6);
    \draw [thick, black] (0,0) -- (0.2,0.65);
    \draw [fill] (0,0) circle [radius=0.01];
    \draw [fill] (0.4,0.6) circle [radius=0.02];
    \draw [fill] (0.31,0.45) circle [radius=0.02];
    \draw [fill] (0.2,0.65) circle [radius=0.02];
    \draw [fill] (0.16,0.5) circle [radius=0.02];
   
    \draw (-0.07,-0.07) node {$0$};
    \draw (0.5,0.7) node {$q^1$};
    \draw (0.1,0.58) node {$q^2$};
    \draw (0.2,0.9) node {$p^2$};
    \draw (0.7,-0.55) node {$p^4$};
    \draw (-0.85,-0.27) node {$p^6$};
    \draw (0.28,0.18) node {$\frac{2}{s+1} q^1$};
    \draw (-0.08,0.35) node {$\frac{2}{s+1} q^2$};
    \draw (0.87,-0.42) node {$\frac{p^3+p^4}{2}$};
    \draw (-0.87,-0.45) node {$\frac{p^5+p^6}{2}$};
    \draw (-0.07,0.93) node {$\frac{p^1+p^2}{2}$};
     \end{tikzpicture}
\caption{
Construction from Example \ref{ex:notopt}, $s=1.5$: $C$ (black), $-C$ (black dashed), $K$, $-K$ (black dotted), $\conv(C\cup(-C))$ (red), $\conv(K\cup(-K))$ (red dotted), 
$\frac{C-C}{2}$ (blue), $\frac{K-K}{2}$ (blue dashed), $C\cap(-C)=K\cap(-K)$ (yellow).
}     
\label{fig:ScapsS_cut}
\end{figure}
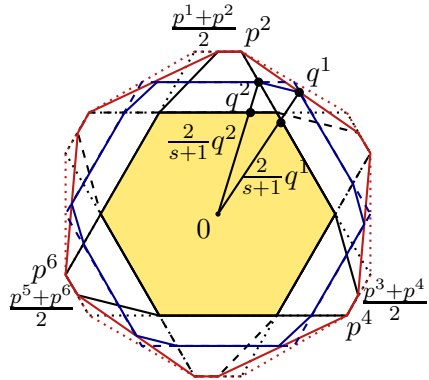

Note that 
\begin{align*}
    \frac{K-K}{2}&= \left(\conv(\left\{\pm \frac{p^1-p^4}{2}, \pm \frac{p^2-p^5}{2}, \pm \frac{p^3-p^6}{2} \right\}\right) \quad \text{and} \\
    K\cap(-K)&=\conv\left(\left\{\pm \frac{p^1-p^4}{s+1}, \pm \frac{p^2-p^5}{s+1}, \pm \frac{p^3-p^6}{s+1} \right\}\right).
\end{align*}
Moreover, $K$ and $C$ are Minkowski centered with $s(K) = s(C)=s$ and
\begin{align*}
\frac{C-C}{2}&= \conv\left(\left\{ \pm \frac{2p^2-p^5-p^6}{4}, \pm \frac{2p^4-p^1-p^2}{4}, \pm \frac{2p^6-p^3-p^4}{4}\right.\right.,\\
&\left.\left.\pm \frac{(s+2)p^2-p^5}{2(s+1)}, \pm \frac{(s+2)p^6-p^3}{2(s+1)}, \pm \frac{(s+2)p^4-p^1}{2(s+1)}\right\}\right).    
\end{align*}
Now, assume that 
\[
\frac{C-C}{2} \subset^{opt} \frac{s+1}{2}\left(\frac{C^\circ-C^\circ}{2}\right)^\circ.
\]
Then there must exist a vertex of $\frac{C-C}{2}$, which also belongs to $\bd \left( \frac{s+1}{2} \left(\frac{C^\circ-C^\circ}{2}\right)^\circ\right)$. Since the vertices of $\frac{C-C}{2}$ have two different types, we have to consider both of them:
let w.l.o.g. $q^1=\frac{2p^2-p^5-p^6}{4}$ and $q^2=\frac{(s+2)p^2-p^5}{2(s+1)}$. Then, for both $i=1,2$, 
\[
\frac{2}{s+1} q^i \in \frac{2}{s+1} \frac{K-K}{2}= K\cap(-K)=C\cap(-C)= \conv\left(\left\{ \pm \frac{p^2-p^5}{s+1}, \pm \frac{p^3-p^6}{s+1}, \pm \frac{p^4-p^1}{s+1}\right\}\right).
\]
Remember, that by Part (ii) of Lemma \ref{lem:Regard_of_Asym} we have $C \cap (-C) \subset^{opt} \left(\frac{C^\circ-C^\circ}{2}\right)^\circ$. Thus we would obtain 
\[
\pm \frac{p^2-p^5}{s+1}, \pm \frac{p^3-p^6}{s+1}, \pm \frac{p^4-p^1}{s+1} \in (C\cap(-C)) \cap \bd \left( \left(\frac{C^\circ-C^\circ}{2}\right)^\circ\right).
\]
If $\frac{2}{s+1} q^1 \in \bd \left( \left(\frac{C^\circ-C^\circ}{2}\right)^\circ\right)$ would be true, then 
\[
\frac{p^2-p^5}{s+1}, \frac{p^3-p^6}{s+1}, \frac{2}{s+1} q^1 \in (C\cap(-C)) \cap \bd \left( \left(\frac{C^\circ-C^\circ}{2}\right)^\circ\right),
\] 
which is only possible if the full edge $[\frac{p^2-p^5}{s+1}, \frac{p^3-p^6}{s+1}] \subset (C\cap(-C)) \cap \bd \left( \left(\frac{C^\circ-C^\circ}{2}\right)^\circ\right)$.
By the symmetries of $C$ we would conclude that $C\cap(-C) = \left(\frac{C^\circ-C^\circ}{2}\right)^\circ$, obtaining a contradiction to Part (iii) of Theorem \ref{thm:reverse_inclusions}. In case $\frac{2}{s+1} q^2 \in \bd \left( \left(\frac{C^\circ-C^\circ}{2}\right)^\circ\right)$, we obtain a similar conclusion (c.f.~Figure \ref{fig:ScapsS_cut}).
Thus, $C$ fulfills $\frac{C-C}{2} \subset  \mathrm{int}(\frac{s+1}{2}\left(\frac{C^\circ-C^\circ}{2}\right)^\circ)$.

\begin{remark}
Let $\omega(s) \in [1,\frac{s+1}{2}]$ be such that $\frac{C-C}{2} \subset^{opt} \omega(s)  \left(\frac{C^\circ - C^\circ}{2}\right)^\circ$, where $s=s(C)$.

Then there exists $x \in \R^n$ such that 
\[
\|x\|_{\frac{C-C}{2}}=\|x\|_{\omega(s)  \left(\frac{C^\circ - C^\circ}{2}\right)^\circ}.
\]

On the one hand, $\frac{x }{\|x\|_{C}} \in \bd(C)$ and $\frac{x }{\|x\|_{-C}}  \in \bd( -C)$, implying $\frac{1}{2} \left( \frac{1 }{\|x\|_{K}}+\frac{1}{\|x\|_{C}} \right) x \in \frac{C-C}{2}$. And thus, we get 
$\|x\|_{\frac{C-C}{2}}\leq
\left(
\frac{\frac{1}{\|x\|_C}+\frac{1}{\|x\|_{-C}}}{2} \right)^{-1}$. On the other hand, $\|x\|_{ \left(\frac{C^\circ - C^\circ}{2}\right)^\circ}=\frac{\|x\|_C+ \|x\|_{-C}}{2}$. Therefore,
\[
\omega(s)=\frac{\|x\|_{ \left(\frac{C^\circ - C^\circ}{2}\right)^\circ}}{\|x\|_{\frac{C-C}{2}}} \geq \frac{\|x\|_C+ \|x\|_{-C}}{2} \frac{\frac{1}{\|x\|_C}+\frac{1}{\|x\|_{-C}}}{2}=\frac{(\|x\|_C+ \|x\|_{-C})^2}{4 \|x\|_C \|x\|_{-C}}. 
\]

Let w.l.o.g. $\|x\|_C \geq \|x\|_{-C}$ and $\rho:=\frac{\|x\|_C}{\|x\|_{-C}}$. Since $C$ is Minkowski centered,
we have $1 \leq \rho \leq s$, and thus

\begin{equation}\label{eq:LBofOmega}
\omega(s) \geq \frac{(\rho+1)^2}{4 \rho}.
\end{equation}
Note, that the right term above
attains its maximum for $\rho=s$. Since we actually get this value, whenever
$x$ is an asymmetry point of $C$, we conclude the assertion.

\end{remark}

\end{example}

\begin{remark}
Let $P \subset \R^n$ be a Minkowski centered polytope. By Proposition \ref{prop:Opt_Containment} there exist $a^i \in \R^n$ and vertices $x^i$, such that $H_{a^i,1}^{\le}$ define facets of $P$, $i \in [k+1]$ and $s(P) \leq k \leq n$,
such that $0 \in \conv(\{a^1,\dots,a^{k+1}\})$ and $-x^i \in H_{a^i,s(P)}$, $i \in [k+1]$.

Now, we do not only have $\frac{P-P}2 \subset^{opt} \frac{(s(P)+1)}2 (P \cap (-P))$ from Theorem \ref{thm:reverse_inclusions}, but also the fact that for any vertex $y$ of the facet $P \cap H_{a^i,1}$ of $P$, we have $\frac{y-x^i}2 \in H_{a^i,\frac{s(P)+1}2} \cap \frac{P-P}2 \cap \frac{s(P)+1}{2} (P \cap (-P))$, i.e.~$\frac{P-P}{2}$ touches $\frac{s(P)+1}{2} (P \cap (-P))$ in all the facets $\frac{s(P)+1}{2} (P \cap (-P)) \cap H_{a^i,\frac{s(P)+1}2}$ with a full facet (c.f.~\cite[Lemma 2.8]{BG}, where this fact is shown for simplices).
\end{remark}

\begin{remark}
It is well known, that $s(K) = \inf_{C \in \K_0^n} d_{BM}(K,C)$  for every $K \in \K^n$ (see, e.g., \cite{Gr}). Furthermore, in \cite[Prop. 3.1]{BrG2} it is shown that this infimum is always attained by $\frac{K-K}{2}$. In general, if $C \in \K^n_0$ and $K \in \K^n$ we see from the definition of the Banach-Mazur distance that
\[
d_{BM}(K,C) = s(K) \Longleftrightarrow \exists \, L \in GL(n), t^1,t^2 \in \R^n \text{ s.t. } -K - t^1 \subset L(C) \subset s(K)K + t^2.
\]
Since $L(C)$ is symmetric, we may symmetrize and replace the right-hand side above by
\[
 \exists \, L \in GL(n), t^1,t^2 \in \R^n \text{ s.t. } \conv((K+t^1) \cup (-K-t^1)) \subset L(C) \subset s(K) ((K + t^2) \cap (-K-t^2)).
\]
For a Minkowski concentric $K$ we now immediately obtain that all four choices
\[
C \in\left\{K\cap(-K),\left(\frac{K^\circ-K^\circ}{2}\right)^\circ,\frac{K-K}{2},\conv(K\cup(-K))\right\}
\]
of symmetrizations of $K$ considered in this paper fulfill $d_{BM}(K,C) = s(K)$ and are therefore minimizers for the Banach-Mazur distance between $K$ and $\K^n_0$.

Moreover, with help of the reverse containments from Theorem \ref{thm:reverse_inclusions} we obtain some upper bounds on the Banach-Mazur distances of pairs of these symmetrizations, e.g.~ 
$d_{BM}(K \cap (-K),\conv(K \cup (-K)) \le s(K)$ or $d_{BM}(K \cap (-K), \frac{K^\circ-K^\circ}{2}) \le \frac{2s(K)}{s(K)+1}$. However, this bounds do not even have to be tight when the containments between those sets are. E.g.~is the Banach-Mazur distance of any two of the symmetrizations of a regular triangle in the plane exactly 1, as they are all regular hexagons.
\end{remark}

\section{Improving the containment factors in the forward direction}\label{sec:improving_inclusions}

Now we prove Theorem \ref{thm:small_asym_no_improve}. 
\begin{proof}[Proof of Theorem \ref{thm:small_asym_no_improve}]
We start showing Part a). 
\begin{enumerate}[(i)]
\item We first show that 
$\alpha_1(s) \geq \frac{2}{s+1}$ independently of $n$. 
By the definition of $\alpha(s)$ and Part (ii) of Theorem \ref{thm:reverse_inclusions} we have $C \cap (-C) \subset^{opt} \alpha(s) \cdot \conv(C \cup (-C)) \subset^{opt} \frac{\alpha(s) (s+1)}{2} \cdot \left(\frac{C^\circ -C^\circ}{2}\right)^{\circ}$, while
Part (ii) of Lemma \ref{lem:Regard_of_Asym} gives us $C \cap(-C) \subset^{opt} \left(\frac{C^\circ - C^\circ}{2}\right)^{\circ}$. Hence, $\alpha(s)$ must be always at least $\frac{2}{s+1}$. 

Next we show $\alpha_1(s)=\frac{2}{s+1}$ in any dimension if $s \le 2$.  
First let $n=2$ and consider $C = S \cap (-s S)$ with $s \in [1,2]$ and the regular triangle $S=\conv \left(\left\{p^1, p^2,p^3 \right\}\right)$. 
Now choose w.l.o.g. the vertex $v$ of $C \cap (-C)$ with $v \in \pos \left(\left\{ p^1, p^2 \right\}\right)$ and let $\mu \ge 1$ be such that $\mu v \in \bd(\conv(C \cup (-C)))$. Let $q$ be a vertex of $\conv(C \cup (-C))$, such that $q \in [p^2, v]$.

\begin{figure}[H]
\centering
    \begin{tikzpicture}[scale=3]
    \draw [thick, black] (-0.11547,0.8) -- (0.11547,0.8) -- (0.75038,-0.2997)-- (0.6348,-0.5)-- (-0.6348,-0.5)-- (-0.75038,-0.2997)--(-0.11547,0.8);
    \draw [thick, dashed] (0.11547,-0.8) -- (-0.11547,-0.8) -- (-0.75038,0.2997)-- (-0.6348,0.5)-- (0.6348,0.5)-- (0.75038,0.2997)--(0.11547,-0.8);
    \draw [thick, black, dotted] (-0.11547,0.8) -- (0,1)--(0.11547,0.8); 
    \draw [thick, black, dotted] (0.75038,-0.2997)-- (0.87,-0.5)--(0.6348,-0.5);
    \draw [thick, black, dotted] (-0.75038,-0.2997)-- (-0.87,-0.5)--(-0.6348,-0.5);
    \draw [thick, gray] (0.75038,0.2997)--(0.75038,-0.2997);
    \draw [thick, gray] (0,0)--(0.86,-0.5);
    \draw [thick, gray] (0,0)--(0.75,0);
    
    \draw [fill] (0,0) circle [radius=0.01];
    \draw [fill] (0.5774,0) circle [radius=0.02];
    \draw [fill] (0.86,-0.50035) circle [radius=0.02];
    \draw [fill] (0,1) circle [radius=0.02];
    \draw [fill] (0.75,0) circle [radius=0.02];
    \draw [fill] (0.75,-0.29) circle [radius=0.02];
    \draw (-0.07,-0.07) node {$0$};
    \draw (0.42,0.1) node {$v$};
    \draw (0.9,0) node {$\mu v$};
    \draw (0.85,-0.28) node {$q$};
    \draw (0,1.1) node {$p^1$};
    \draw (1,-0.55) node {$p^2$};
     \end{tikzpicture}
\caption{Construction from the proof of Part (i) of Theorem \ref{thm:small_asym_no_improve}: $C=S \cap (-s S)$ ($s=s(C)=1.5$) (black), $-C$ (dashed).
}
\end{figure}
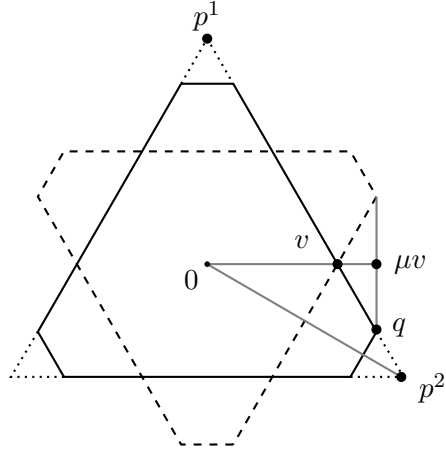
Since $\|p^2\|=1$, we have 
\[ 	
\left\|v-q \right\|=\frac{\left \|p^2-\left(\frac12 p^2 \right)\right\|-\left \|p^2-\left(\frac{s}{2}p^2 \right)\right\|}{\cos  (\pi/6) }= \frac{2}{\sqrt{3}} \left( \frac s2 -\frac12\right). 
\]
Thus, since $\|v\|=\frac{1}{\sqrt{3}}$, we have 
\[
\frac{\left\|\mu v \right\|}{\left\|v \right\|}= 1+ \frac{\left\|\mu v -v\right\|}{\left\|v \right\|}= 1+ \frac{\left\|v-q \right\| \sin  (\pi/6)}{\left\|v \right\|}=1+ \frac{ \frac{1}{2} \left\|v-q \right\| }{\frac{1}{\sqrt{3}}}= 1+\frac s2 -\frac1n,   
\]
which implies
\[
\alpha(s)= \frac{1}{1+\frac s2 -\frac12} =\frac{2}{s+1}. 
\]

For $n\geq 3$ we can simply embed the above $C$ keeping the Minkowski center to be still 0 into $n$-space. This keeps its asymmetry value and also the correct factor for the containment between $C \cap (-C)$ and $\conv(C \cup (-C))$. 

\item Obviously, from Proposition \ref{prop:means_of_sets} we know that $\alpha(s) \le 1$.

We now show that for any $s \ge \varphi$ 
there exists a Minkowski centered $C$ with $s(C)=s$ for which $\alpha(s)=\frac{s}{s^2-1}$.

For any $s \ge \varphi$ we define
\[C=\conv\left(\left\{ 
\begin{pmatrix} \pm 1 \\ (\varphi+1)(2-s-\frac{1}{s+1})\end{pmatrix}, 
\begin{pmatrix} \pm 1 \\ - \frac{\varphi+1}{s+1} \end{pmatrix}, 
\begin{pmatrix} 0 \\ s\frac{\varphi+1}{s+1} \end{pmatrix}\right\}\right).\]

Since $-(1/s) C \subset C$ with
\[
-\frac1s \begin{pmatrix} \pm 1 \\ - \frac{\varphi+1}{s+1} \end{pmatrix} \in 
\left[\begin{pmatrix}  \mp 1 \\ (\varphi+1)(2-s-\frac{1}{s+1})\end{pmatrix} , \begin{pmatrix} 0 \\ s\frac{\varphi+1}{s+1} \end{pmatrix}\right]
\]
and
\[
-\frac1s \begin{pmatrix} 0 \\ s\frac{\varphi+1}{s+1} \end{pmatrix} \in 
\left[\begin{pmatrix} 1 \\ - \frac{\varphi+1}{s+1} \end{pmatrix}, \begin{pmatrix}  -1 \\ - \frac{\varphi+1}{s+1} \end{pmatrix}\right],
\]
we obtain from Proposition \ref{prop:Opt_Containment} that $C$ is Minkowski centered and $s(C)=s$. 

Let $\alpha(s)<1$ be such that $C \cap (-C) \subset^{opt} \alpha(s) \cdot \conv(C \cup (-C))$. Then due to the symmetries of $C$, we have 
\[
\alpha(s)=\max \left\{ \frac{\|v\|}{\|w\|}, \frac{\|p\|}{\|q\|}  \right\},
\]
where $v$ and $p$ are vertices of $C \cap (-C)$, such that $v = \begin{pmatrix} \frac{s}{s^2-1} \\ 0\end{pmatrix}$ and  
\[p=\left[\begin{pmatrix} 0 \\ s\frac{\varphi+1}{s+1} \end{pmatrix},\begin{pmatrix} 1 \\ (\varphi+1)(2-s-\frac{1}{s+1})\end{pmatrix}\right] \cap \left[\begin{pmatrix} 1 \\  \frac{\varphi+1}{s+1} \end{pmatrix},\begin{pmatrix} -1 \\  \frac{\varphi+1}{s+1} \end{pmatrix}\right],
\]
while $w$ and $q$ are rescalations of $v$ and $p$, respectively, belonging to $\bd(\conv(C \cup (-C))).$
Now, we see 
\[w=\frac12 \left( \begin{pmatrix} 1 \\ - \frac{\varphi+1}{s+1} \end{pmatrix}+ \begin{pmatrix} 1 \\  \frac{\varphi+1}{s+1} \end{pmatrix} \right)=\begin{pmatrix} 1 \\ 0\end{pmatrix} \in \bd(\conv(C \cup (-C))),
\] 
which shows $w=\frac{s^2-1}{s} v$.
Note that for some $x \in \R$ and some $\lambda \in [0,1]$, we have


\[
p=\begin{pmatrix} x \\ \frac{\varphi+1}{s+1} \end{pmatrix}=\lambda  \begin{pmatrix} 0  \\ \frac{s(\varphi+1)}{s+1} \end{pmatrix} +(1-\lambda)  \begin{pmatrix} 1  \\ \frac{(\varphi+1)(-s^2+s+1)}{s+1}
\end{pmatrix}. 
\]
Thus, $\lambda=\frac{s}{s+1}$ and
\[
p = \frac{1}{s+1}\begin{pmatrix}1 \\ \varphi+1 \end{pmatrix}.
\]

Now let $q=\nu p  \in \bd(\conv(C \cup (-C)))$ for some $\nu >1$. Then  
\[
q=\lambda  \begin{pmatrix} 0  \\ \frac{s(\varphi+1)}{s+1} \end{pmatrix} +(1-\lambda) \begin{pmatrix} 1  \\ \frac{\varphi+1}{s+1}
\end{pmatrix}. 
\]
Thus, $\lambda=\frac12$ and 
\[
q = \frac12 \begin{pmatrix} 1  \\ \varphi+1 \end{pmatrix}.
\]
Let $\alpha(s)<1$ be such that $C \cap (-C) \subset^{opt} \alpha(s) \cdot \conv(C \cup (-C))$. Then due to the symmetries of $C$, we have 
\[
\alpha(s)=\max \left\{ \frac{\|v\|}{\|w\|}, \frac{\|p\|}{\|q\|}  \right\} = \max \left\{ \frac{2}{s+1}, \frac{s}{s^2-1} \right\} = \frac{s}{s^2-1}.
\]


By the definition of $\gamma(n)$, we have $\alpha_2(S) = 1$ for $s \leq \gamma(n)$, while $\alpha_2(s) \le \psi \frac{n}{n+1}$ if $s > \gamma_2(n)$ follows from Part (i) of Theorem 1.5. 
\end{enumerate}

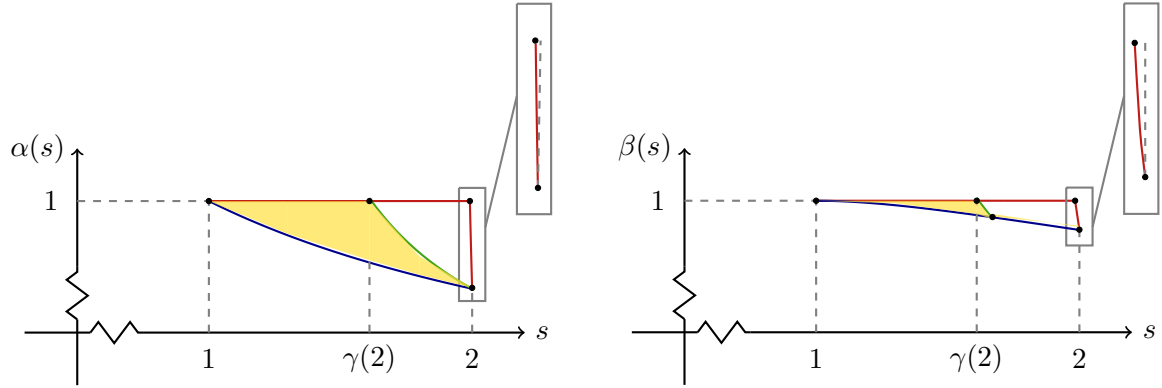
\begin{figure}[ht]
  \begin{subfigure}[b]{0.47\textwidth}
    \centering
  \begin{tikzpicture}[scale=3.5]
    \draw[thick, discont] (0.05,0) -- (0.25,0);
    \draw[thick, discont] (0,0.05) -- (0,0.25);
    \draw [thick] (-0.2,0) -- (0.05,0);
    \draw [thick] (0,-0.2) -- (0,0.05);
    \draw[->] [thick] (0.25, 0) -- (1.7, 0) node[right] {$s$};
    \draw[->] [thick] (0, 0.25) -- (0, 0.7); 
    \draw [thick,gray, shift={(0,-0.5)}] (1.67,1.75) -- (1.8,1.75)--(1.8,0.95)--(1.67,0.95)--(1.67,1.75);
   \draw [thick,gray, shift={(0,-0.5)}] (1.45,1.05) -- (1.55,1.05)--(1.55,0.62)--(1.45,0.62)--(1.45,1.05);
    \draw [thick,gray, shift={(0,-0.5)}] (1.55,0.9) -- (1.67,1.4);
    \draw [thick, dred,shift={(-0.5,-0.5)}] (1,1) -- (1.61,1);
    \draw [thick, dred,shift={(-0.5,-0.5)}] (1.61,1) -- (1.9914,1);
    \draw[thick, dred, domain=1.9914:2, smooth, variable=\x,shift={(-0.5,-0.5)}]  plot ({\x},
    {(-26*\x^2+36*\x+34)/(18*\x^2-24*\x-21)});
    \draw[thick, dred, domain=1.9914:2, smooth, variable=\x,shift={(-0.25,-0.6)}]  plot ({\x},
    {1.7*(-26*\x^2+36*\x+34)/(18*\x^2-24*\x-21)});
    \draw[thick, dblue, domain=1:2, smooth, variable=\x, dblue,shift={(-0.5,-0.5)}]  plot ({\x}, {2/(\x+1)});
    \draw[thick, dgreen, domain=1.61:2, smooth, variable=\x,shift={(-0.5,-0.5)}]  plot ({\x}, {(\x)/((\x)^2-1)});
    \fill [fill=lgold, fill opacity=0.7, domain=1:1.3, variable=\x,shift={(-0.5,-0.5)}] (1,1) -- (1.3,1)-- (1.3,0.875)--(1,1);
    \fill [fill=lgold, fill opacity=0.7, domain=1.3:1.61, variable=\x,shift={(-0.5,-0.5)}] (1.3,1) -- (1.61,1)-- (1.61,0.775)--(1.3,0.875);
    \fill [fill=lgold, fill opacity=0.7, domain=1.3:1.61, variable=\x,shift={(-0.5,-0.5)}] (1.61,1) -- (1.8,0.79)--(1.8,0.72)-- (1.61,0.775)--(1.61,1);
    \fill [fill=lgold, fill opacity=0.7, domain=1.8:2, variable=\x,shift={(-0.5,-0.5)}] (1.8,0.79) -- (2,0.67)--(1.8,0.72)-- (1.8,0.79)--(1.8,0.79);
    
    \draw [thick, dashed,gray] (0,0.5) -- (0.5,0.5);
    \draw [thick, dashed,,gray] (1.11,0)--(1.11,0.28);
    \draw [thick, dashed,gray] (1.5,0) -- (1.5,0.17);
    \draw [thick, dashed,gray] (0.5,0) -- (0.5,0.5);
    \draw [thick, dashed,gray] (1.75,0.55) -- (1.76,1.11);
    
    \draw [fill,shift={(-0.5,-0.5)}] (1,1) circle [radius=0.01];
    \draw [fill,shift={(-0.5,-0.5)}] (1.61,1) circle [radius=0.01];
    \draw [fill,shift={(-0.5,-0.5)}] (2,0.67) circle [radius=0.01];
    \draw [fill,shift={(-0.5,-0.5)}] (1.9914,1) circle [radius=0.01];
    \draw [fill] (1.74,1.11) circle [radius=0.01];
    \draw [fill] (1.75,0.55) circle [radius=0.01];
    
    \draw (-0.15,0.7) node {$\alpha(s)$};
    \draw (-0.1,0.5) node {$1$};
    \draw (0.5,-0.1) node {$1$};
    \draw (1.11,-0.1) node {$\gamma(2)$};%
    \draw (1.5,-0.1) node {$2$};
     \end{tikzpicture}
    \caption*{$n=2$: $\alpha_2(s) = 1$ for $s \leq \gamma(2)$ (red). For $s \ge \gamma(2)$ we have $\alpha_2 (s) \ge \frac{s}{s^2-1}$ (green), while $\alpha_2(s) < 1$ for $s > \gamma(2)$ and $\alpha_2 \le \frac{2}{3} \psi(2)  = \frac{-26s^2+36s+34}{18s^2-24s-21}$ for $s > \gamma_2(2)$ (red). The lower bound is $\alpha_1(s) = \frac{2}{s+1}$ (blue).}
    \label{fig:GH}
    \begin{minipage}{.1cm}
   \vspace{1em}
    \end{minipage}
  \end{subfigure} \hfill
  \begin{subfigure}[b]{0.47\textwidth}
    \centering
        \begin{tikzpicture}[scale=3.5]
    \draw[thick, discont] (0.05,0) -- (0.25,0);
    \draw[thick, discont] (0,0.05) -- (0,0.25);
    \draw [thick] (-0.2,0) -- (0.05,0);
    \draw [thick] (0,-0.2) -- (0,0.05);
    \draw[->] [thick] (0.25, 0) -- (1.7, 0) node[right] {$s$};
    \draw[->] [thick] (0, 0.25) -- (0, 0.7);  
    \draw [thick, dred] (0.5,0.5) -- (1.21,0.5);
    \draw [thick, dred] (1.21,0.5) -- (1.4837,0.5);
    \draw [thick, dred] (1.4837,0.5) -- (1.5,0.39);
  
    \draw [thick,gray, shift={(0,-0.5)}] (1.67,1.75) -- (1.8,1.75)--(1.8,0.95)--(1.67,0.95)--(1.67,1.75);
   \draw [thick,gray, shift={(0,-0.5)}] (1.45,1.05) -- (1.55,1.05)--(1.55,0.83)--(1.45,0.83)--(1.45,1.05);
    \draw [thick,gray, shift={(0,-0.5)}] (1.55,0.9) -- (1.67,1.4);
    \draw[thick, dred] plot [smooth] coordinates { (1.71,1.1) (1.73,0.77) (1.75,0.59) };
    \draw [thick, dashed,gray] (1.75,1.1) -- (1.75,0.59);
    \draw [fill] (1.71,1.1) circle [radius=0.01];
    \draw [fill] (1.75,0.59) circle [radius=0.01];
    
    \draw[thick, dblue, domain=1:2, smooth, variable=\x,shift={(-0.5,-0.5)}]  plot ({\x}, {(4*\x)/(\x+1)^2)});
    \draw [thick, dashed,gray] (0,0.5) -- (0.5,0.5);
    \draw [thick, dashed,gray] (1.11,0) -- (1.11,0.5);
    \draw [thick, dashed,gray] (1.5,0) -- (1.5,0.39);
    \draw [thick, dashed,gray] (0.5,0) -- (0.5,0.5);
    \draw[thick, dgreen, domain=1.61:1.67, smooth, variable=\x,shift={(-0.5,-0.5)}]  plot ({\x}, {(\x)/((\x)^2-1)});
    \fill [fill=lgold, fill opacity=0.7, domain=1.61:1.67, variable=\x,shift={(-0.5,-0.5)}] (1.61,1)--(1.67,0.93) --(1.61, 0.95); 
    \fill [fill=lgold, fill opacity=0.7, domain=1:1.3, variable=\x,shift={(-0.5,-0.5)}] (1,1) -- (1.3,1)--(1.3,0.99)--(1,1);
    \fill [fill=lgold, fill opacity=0.7, domain=1.3:1.61, variable=\x,shift={(-0.5,-0.5)}] (1.3,1) -- (1.61,1)--(1.61,0.95)--(1.3,0.99);
    \fill [fill=lgold, fill opacity=0.7, domain=1.61:2, variable=\x,shift={(-0.5,-0.5)}] (1.67,0.94) --(2, 0.905)--(2, 0.9)--(1.67,0.95);
    
    
    \draw [fill] (0.5,0.5) circle [radius=0.01];
    \draw [fill] (1.11,0.5) circle [radius=0.01];
    \draw [fill] (1.5,0.39) circle [radius=0.01];
    \draw [fill] (1.4837,0.5) circle [radius=0.01];
   
    \draw [fill] (1.17,0.438) circle [radius=0.01];
    
    \draw (-0.15,0.7) node {$\beta(s)$};
    \draw (-0.1,0.5) node {$1$};
    \draw (0.5,-0.1) node {$1$};
    \draw (1.11,-0.1) node {$\gamma(2)$};
    \draw (1.5,-0.1) node {$2$};
     \end{tikzpicture}
    \caption*{$n=2$: $\beta_2(s) = 1$ for $s \leq \gamma(2)$ (red); for $s \in [\varphi,2]$ the upper bound is at least $\max \{ \frac{s}{s^2-1}, \frac{4s}{(s+1)^2} \}$ (green) and for $s \in [\gamma_2(n),2]$ also at most 
    $ \frac{8}{9} \mu \psi$
    (red). The lower bound is 
    $\frac{4s}{(s+1)^2}$, which is sharp for $s \in [1,2]$ (blue).} \label{fig:GH-symm}
  \begin{minipage}{.1cm}
   \vspace{1em}
    \end{minipage}
  \end{subfigure}

\caption{Region of possible values of parameters $\alpha(s)$, $\beta(s)$ with $s \in [1,n]$, s.th. $C \cap (-C) \subset^{opt} \alpha(s) \conv(C \cup (-C))$, $\left(\frac12 (C^\circ - C^\circ)) \right)^{\circ}\subset^{opt} \beta(s) \, \frac12 (C-C)$ for some Minkowski centered $C \in\K^n$ with $s(C)=s$ from Theorem \ref{thm:small_asym_no_improve}.}
\end{figure}

We now proceed with Part b). 
\begin{enumerate}[(i)]
\item
Let $\beta(s) \leq 1$ be such that
\[
\left(\frac{C^\circ - C^\circ}{2} \right)^{\circ} \subset^{opt} \beta(s) \frac{C-C}{2}.
\]
On the one hand, by Part (iv) of Theorem \ref{thm:reverse_inclusions} we have
\[
\left(\frac{C^\circ - C^\circ}{2} \right)^{\circ} \subset^{opt} \beta(s) \frac{C-C}{2} \subset^{opt} \beta(s) \frac{s+1}{2} C \cap (-C).
\]
On the other hand, from Part (iii) of Theorem \ref{thm:reverse_inclusions} we know that 
\[
\left(\frac{C^\circ - C^\circ}{2} \right)^{\circ} \subset^{opt} \frac{2s}{s+1} C \cap (-C).
\]
Hence $ \frac{2s}{s+1}  \leq \beta(s) \frac{s+1}{2}$, 
which implies $\beta_1(s) \geq \frac{4s}{(s+1)^2}$.


Now, consider the hexagon 
\[C=\conv\left(\left\{ 
\begin{pmatrix} \pm  \frac{\sqrt{3}}{3} \left( 1- \frac{s}{2} \right) \\ \frac{s}{2}  \end{pmatrix}, 
\begin{pmatrix} \pm  \frac{\sqrt{3}}{3} \left( \frac{s+1}{2} \right) \\  \frac{1}{2}-\frac{3s}{4} \end{pmatrix}, \begin{pmatrix} \pm \frac{\sqrt{3}}{3} \left( s- \frac{1}{2} \right) \\  -\frac{1}{2} \end{pmatrix} \right\}\right)
\]
with $s\in[1,2]$. For all $s\in[1,2]$ we have
\begin{align*}
- \frac1s \begin{pmatrix}  \frac{\sqrt{3}}{3} \left( 1- \frac{s}{2} \right) \\ \frac{s}{2}  \end{pmatrix} &=  \frac{2(s^2-1)}{s(4s-1)}   \begin{pmatrix}  -\frac{\sqrt{3}}{3} \left( s- \frac{1}{2} \right) \\ -\frac{1}{2}  \end{pmatrix} + \left( 1-\frac{2(s^2-1)}{s(4s-1)}\right) \begin{pmatrix}  \frac{\sqrt{3}}{3} \left( s- \frac{1}{2} \right) \\ -\frac{1}{2}  \end{pmatrix}, 
\end{align*}
as well as
\begin{align*}
- \frac1s \begin{pmatrix}  \pm \frac{\sqrt{3}}{3} \left( s- \frac{1}{2} \right) \\ -\frac{1}{2}  \end{pmatrix} &= \frac{3s^2-2s+2}{s(5s-2)} \begin{pmatrix}  \frac{\sqrt{3}}{3} \left( 1- \frac{s}{2} \right) \\ \frac{s}{2}  \end{pmatrix} + \left(1-\frac{3s^2-2s+2}{s(5s-2)}\right) \begin{pmatrix}  \frac{\sqrt{3}}{3} \left( \frac{s+1}{2} \right) \\ \frac{1}{2}-\frac{3s}{4}  \end{pmatrix},
\end{align*}
with $\frac{2(s^2-1)}{s(4s-1)}, \frac{3s^2-2s+2}{s(5s-2)} \in [0,1]$.
Hence, $C$ is Minkowski centered with $s(C)=s$. 
Since $C$ is a hexagon with 3 pairs of parallel edges, it 
turns out that its arithmetic mean stays to be a hexagon:
\[
\frac{C-C}{2}=\conv\left(\left\{ 
\begin{pmatrix} \pm  \frac{\sqrt{3}}{12} \left( s+1 \right) \\ \frac{s+1}{4}  \end{pmatrix}, \begin{pmatrix} \pm  \frac{\sqrt{3}}{12} \left( s+1 \right) \\ -\frac{s+1}{4}  \end{pmatrix}
\begin{pmatrix} \pm  \frac{\sqrt{3}}{3} \left( \frac{s+1}{2} \right) \\  0 \end{pmatrix} \right\}\right).
\]
The next step to do is to calculate $C^\circ=: \conv(\{q^1, \dots, q^6\})$. Since the vertices of $C$ are the outer normals of the edges of $C^\circ$ and $0 \in \inter (C)$, we obtain the vertices of $C^\circ$ as the solution of pairs of inequalities of the form $(v^i)^T x = 1$, where the $v^i$'s are consecutive vertices of $C$. 
Moreover, we make use of the fact that $C$, and therefore also $C^\circ$, is symmetric w.r.t.~the $y$-axis. Hence, it suffices to calculate four of the vertices of $C^\circ$.
Let $q^1$ fulfill the equations
\begin{align*}
 (q^1)^T \begin{pmatrix} \pm \frac{\sqrt{3}}{3} \left(1-\frac{s}{2}\right) \\ \frac{s}{2} \end{pmatrix} = 1,
\end{align*}
 which obviously needs
$q^1=\begin{pmatrix} 0 \\  \frac{2}{s} \end{pmatrix}$.
For $q^2$ we demand
\begin{align*}
 (q^2)^T \begin{pmatrix} \frac{\sqrt{3}}{3} \left(1-\frac{s}{2}\right)  \\ \frac{s}{2} \end{pmatrix}=(q^2)^T \begin{pmatrix} \frac{\sqrt{3}}{3} \left(\frac{s+1}{2}\right)\\ \frac{1}{2}-\frac{3s}{4} \end{pmatrix} = 1.
\end{align*}
and obtain $q^2=\begin{pmatrix} \sqrt{3} \\ 1 \end{pmatrix}$.
The third vertex $q^3$ should fulfill
\begin{align*}
 (q^3)^T \begin{pmatrix} \frac{\sqrt{3}}{3} \left(\frac{s+1}{2}\right)\\ \frac{1}{2}-\frac{3s}{4} \end{pmatrix} =(q^3)^T \begin{pmatrix} \frac{\sqrt{3}}{3} \left(s-\frac{1}{2}\right)  \\ -\frac{1}{2} \end{pmatrix}= 1.
\end{align*}
This leads to $q^3=\begin{pmatrix}  \frac{\sqrt{3}}{s} \\ -\frac{1}{s} \end{pmatrix}$.
Finally, for $q^4$ we have to solve 
\begin{align*}
 (q^4)^T \begin{pmatrix} \pm \frac{\sqrt{3}}{3} \left(s-\frac{1}{2}\right)  \\ -\frac{1}{2} \end{pmatrix} = 1,
\end{align*}
which gives $q^4=\begin{pmatrix} 0 \\  -2 \end{pmatrix}$.

Altogether, we obtain
\[
C^\circ=\conv\left(\left\{ 
\begin{pmatrix} 0 \\  \frac{2}{s} \end{pmatrix}, 
\begin{pmatrix} \pm \sqrt{3} \\ 1 \end{pmatrix}, 
\begin{pmatrix} \pm  \frac{\sqrt{3}}{s} \\ -\frac{1}{s} \end{pmatrix}, \begin{pmatrix} 0 \\  -2 \end{pmatrix} \right\}\right).
\]
Since $C^\circ$ has no parallel edges, $\frac{C^\circ-C^\circ}{2}$ is a 12-gon that computes to
\[
\frac{C^\circ-C^\circ}{2} = \conv\left(\left\{ 
\begin{pmatrix} 0 \\ \pm \frac{s+1}{s} \end{pmatrix}, 
\begin{pmatrix} \pm \frac{\sqrt{3}}{2} \\ \pm \frac32 \end{pmatrix}, 
\begin{pmatrix} \pm  \sqrt{3} \\ 0 \end{pmatrix}, \begin{pmatrix} \pm \sqrt{3} \frac{s+1}{2s}  \\  \pm \frac{s+1}{2s} \end{pmatrix} \right\}\right).
\]
Now let $\beta(s)>1$ be such that 
\[
\left(\frac{C^\circ - C^\circ}{2} \right)^{\circ} \subset^{opt} \beta(s) \frac{C-C}{2}.
\]
Note that 6 vertices of $\frac{C^\circ-C^\circ}{2}$ are rescales of the outer-normals of $\frac{C-C}{2}$, thus the 12-gon $\left(\frac{C^\circ - C^\circ}{2} \right)^{\circ}$ has 6 edges, which are parallel to the corresponding edges of $\frac{C-C}{2}$, thus $v \in  \bd (\beta(s) \frac{C - C}{2})$ for all vertices $v$ of $\left(\frac{C^\circ - C^\circ}{2} \right)^{\circ}$.

So we choose the vertex $v$ that fulfills the equations
\begin{align*}
 v^T \begin{pmatrix} 0 \\ \frac{s+1}{s} \end{pmatrix} = 
 v^T \begin{pmatrix}  \frac{\sqrt{3}}{2} \\ \frac32 \end{pmatrix}= 1,
\end{align*}
which gives $v=\begin{pmatrix} \frac{2-s}{\sqrt{3}(s+1)} \\ \frac{s+1}{s} \end{pmatrix}$. 

Since $v$ is contained in the edge of $\left(\frac{C^\circ - C^\circ}{2} \right)^{\circ}$ with the outer-normal $\begin{pmatrix}  0 \\ 1 \end{pmatrix}$, thus $\frac1{\beta(s)}v$ is contained in the edge of $\frac{C - C}{2}$ with the same outer-normal and 

\[
\frac{1}{\beta(s)} \begin{pmatrix} \frac{2-s}{\sqrt{3}(s+1)} \\ \frac{s+1}{s} \end{pmatrix} = \lambda \begin{pmatrix} -\frac{s+1}{4 \sqrt{3}} \\ \frac{s+1}{4} \end{pmatrix} + (1-\lambda) \begin{pmatrix} \frac{s+1}{4 \sqrt{3}} \\ \frac{s+1}{4} \end{pmatrix}, 
\]
from which we obtain $\lambda=\frac{s-1}{s}$ and $\beta(s)=\frac{4s}{(s+1)^2}$, thus proving that $\beta_1(s) = \frac{4s}{(s+1)^2}$ for all $s \in [1,2]$.

For higher dimensions we can simply embed the above construction in a way keeping the Minkowski center 0.

\item  By the definition of $\gamma(n)$, we have $\beta_2(s)=1$ for $s \leq \gamma(n)$ while $\beta_2(s) \le \mu \psi \frac{n(n+2)}{(n+1)^2}$ if $s > \gamma_3(n)$ follows from Part (ii) of Theorem 1.5.

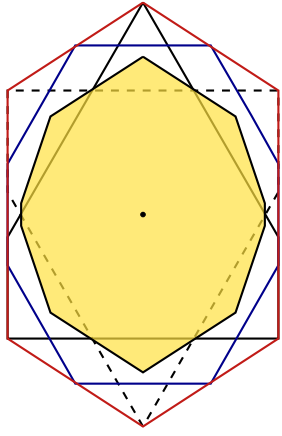
\begin{figure}[H]
\centering
    \begin{tikzpicture}[scale=3]
    \draw [thick] (0,0.94) -- (0.6,-0.1)-- (0.6,-0.55)--(-0.6,-0.55)--(-0.6,-0.1)--(0,0.94);
    \draw [thick, rotate around={180:(0,0)}, dashed] (0,0.94) -- (0.6,-0.1)-- (0.6,-0.55)--(-0.6,-0.55)--(-0.6,-0.1)--(0,0.94);
     \draw [thick, dblue] (0.6,0.225) -- (0.6,-0.225)--(0.3,-0.75)--(-0.3,-0.75)--(-0.6,-0.225) -- (-0.6,0.225)--(-0.3,0.75)--(0.3,0.75)--(0.48,0.44)--(0.6,0.225);
   \draw [thick, fill=lgold, fill opacity=0.7] (0,0.7) --(0.41,0.435)--(0.54,0.05)--(0.54,-0.05)--(0.41,-0.435)--(0,-0.7)--(-0.41,-0.435)--(-0.54,-0.05)--(-0.54,0.05)--(-0.41,0.435)--(0,0.7);
   \draw [thick, dred] (0,0.94) --(0.6,0.55)--(0.6,-0.55)--(0,-0.94)--(-0.6,-0.55)--(-0.6,0.55)--(0,0.94); 
   \draw [fill] (0,0) circle [radius=0.01];  
   \end{tikzpicture}
\caption{Construction from the proof of Part b) (ii) of Theorem \ref{thm:small_asym_no_improve} for 
$s=1.7$: $C$ (black) and $-C$ (black dashed),
$\left(\frac{C^\circ - C^\circ}{2} \right)^{\circ}$ (yellow), 
$\frac{C - C}{2}$ (blue) and $\conv(C\cup(-C))$ (red). 
}
\label{fig:ScapsS}
\end{figure}

In order to show an upper bound on $\beta_2$, 
which is strictly smaller than 1 for $s \in [\varphi,2]$, we first provide the following construction 
in $\R^2$.

Let $C=\conv\left(\left\{ 
\begin{pmatrix} \pm \frac{\sqrt{3}}{2} (s-1) \\ \frac32 \left( -\frac{1}{s+1}+2-s \right) \end{pmatrix}, 
\begin{pmatrix} \pm \frac{\sqrt{3}}{2} (s-1) \\  -\frac{3}{2(s+1)} \end{pmatrix}, 
\begin{pmatrix} 0 \\ \frac{3s}{2(s+1)} \end{pmatrix}\right\}\right)$, $s \in [\varphi,2]$.

Since 
\begin{align*}
\begin{pmatrix} 0 \\ \frac{3}{2(s+1)} \end{pmatrix} &\in 
\left[\begin{pmatrix} -\frac{\sqrt{3}}{2} (s-1) \\  -\frac{3}{2(s+1)} \end{pmatrix}, \begin{pmatrix} \frac{\sqrt{3}}{2} (s-1) \\  -\frac{3}{2(s+1)} \end{pmatrix}\right] \quad \text{and} \\
\begin{pmatrix} \pm \frac{\sqrt{3}}{2} (s-1) \\  -\frac{3}{2(s+1)} \end{pmatrix} &\in 
\left[\begin{pmatrix} \mp \frac{\sqrt{3}}{2} (s-1) \\ \frac32 \left( -\frac{1}{s+1}+2-s \right) \end{pmatrix}, \begin{pmatrix} 0 \\ \frac{3s}{2(s+1)} \end{pmatrix}\right],
\end{align*}
we obtain that $C$ is Minkowski centered with $s(C)=s$.

Note that $C$ has 5 edges, 2 of which are parallel. This implies that $\frac{C-C}{2}$ has 4 different pairs of parallel edges and 8 vertices. Calculating them gives
\[
\frac{C-C}{2}= \conv \left(\left\{ 
\begin{pmatrix} \pm \frac{\sqrt{3}}{2} (s-1) \\  \frac{3}{4}  (2-s) \end{pmatrix}, \begin{pmatrix} \pm \frac{\sqrt{3}}{2} (s-1) \\ - \frac{3}{4}  (2-s) \end{pmatrix}, 
\begin{pmatrix} \pm \frac{\sqrt{3}}{4} (s-1) \\  \frac{3}{4} \end{pmatrix}, \begin{pmatrix} \pm \frac{\sqrt{3}}{4} (s-1) \\  - \frac{3}{4} \end{pmatrix} \right\}\right). 
\]


Next, we determine $C^\circ =: \conv(\{q^1, q^2, q^3, q^4, q^5\})$.  Since the vertices of $C$ are the outer normals of the edges of $C^\circ$ and $0 \in \inter (C)$, we obtain the vertices of $C^\circ$ as the solution of pairs of inequalities of the form $(v^i)^T x = 1$, such that the corresponding $v^i$ are consecutive vertices of $C$. Since $C$ is symmetric w.r.t.~the $y$-axis, so is $C^\circ$. Hence it suffices to calculate $q^1, q^2, q^3$.

Let $q^1$ be such that
\begin{align*}
 (q^1)^T \begin{pmatrix} 0 \\ \frac{3s}{2(s+1)} \end{pmatrix} = 
 (q^1)^T \begin{pmatrix}  \frac{\sqrt{3}}{2} (s-1) \\ \frac32 \left( -\frac{1}{s+1}+2-s \right) \end{pmatrix}= 1.
\end{align*}
This gives $q^1=\begin{pmatrix} \frac{2}{\sqrt{3}} \left( \frac{s+1}{s} \right) \\ \frac{2}{3} \left( \frac{s+1}{s} \right) \end{pmatrix}$. Now, assume $q^2$ be such that 
\begin{align*}
 (q^2)^T \begin{pmatrix}  \frac{\sqrt{3}}{2} (s-1) \\  -\frac{3}{2(s+1)} \end{pmatrix} = 
 (q^2)^T \begin{pmatrix} \frac{-\sqrt{3}}{2} (s-1) \\  -\frac{3}{2(s+1)} \end{pmatrix} = 1.
\end{align*}
We obtain $q^2=\begin{pmatrix} 0 \\ -\frac{2}{3} \left( s+1 \right) \end{pmatrix}$. For $q^3$ we assume 
\begin{align*}
 (q^3)^T \begin{pmatrix}  \frac{\sqrt{3}}{2} (s-1) \\  \frac32 \left( -\frac{1}{s+1}+2-s \right) \end{pmatrix} = 
 (q^3)^T \begin{pmatrix} \frac{\sqrt{3}}{2} (s-1) \\  -\frac{3}{2(s+1)} \end{pmatrix}= 1.
\end{align*}
Then $q^3=\begin{pmatrix}  \frac{2}{\sqrt{3}(s-1)} \\ 0 \end{pmatrix}$. 
Thus, 
\[
C^\circ= \conv\left(\left\{ \begin{pmatrix} \pm \frac{2}{\sqrt{3}} \left( \frac{s+1}{s} \right) \\ \frac{2}{3} \left( \frac{s+1}{s} \right) \end{pmatrix}, 
\begin{pmatrix} 0 \\ -\frac{2}{3} \left( s+1 \right) \end{pmatrix}, 
 \begin{pmatrix} \pm  \frac{2}{\sqrt{3}(s-1)} \\ 0 \end{pmatrix}
\right\}\right).
\]
Note that $C^\circ$ has 5 edges, none of which are parallel, thus $\frac{C^\circ-C^\circ}{2}$ must have 5 different pairs of parallel edges and 10 vertices. We obtain
\begin{align*}
\frac{C^\circ-C^\circ}{2} = \conv\left(\left\{ \begin{pmatrix} \pm \frac{1}{\sqrt{3}}\left( \frac{s+1}{s} \right) \\  \frac{1}{3} \frac{(s+1)^2}{s} \end{pmatrix}, \begin{pmatrix} \pm \frac{1}{\sqrt{3}}\left( \frac{s+1}{s} \right) \\ - \frac{1}{3} \frac{(s+1)^2}{s} \end{pmatrix}, \begin{pmatrix} \pm \frac{2}{\sqrt{3}}\left( \frac{s+1}{s} \right) \\ 0 \end{pmatrix} \right.\right. , \\
 \left.\left.\begin{pmatrix} \pm \frac{1}{\sqrt{3}} \frac{s^2+s-1}{s(s-1)} \\ \frac{1}{3}  \frac{s+1}{s}  \end{pmatrix}, \begin{pmatrix} \pm \frac{1}{\sqrt{3}} \frac{s^2+s-1}{s(s-1)} \\ - \frac{1}{3}  \frac{s+1}{s}  \end{pmatrix} \right\}\right).     
\end{align*}
The last set to be calculated is $\left(\frac{C^\circ - C^\circ}{2} \right)^{\circ} =: \conv\left(\left\{v^1, \dots, v^{10}\right\}\right)$.

Let $v^1$ be such that 
\begin{align*}
 (v^1)^T \begin{pmatrix}  \frac{1}{\sqrt{3}} \frac{s+1}{s} \\  \frac{1}{3} \frac{(s+1)^2}{s}\end{pmatrix}= 1 \quad\text{and}\quad
 (v^1)^T \begin{pmatrix} \frac{1}{\sqrt{3}} \frac{s^2+s-1}{s(s-1)} \\ \frac{1}{3} \frac{s+1}{s} \end{pmatrix}= 1.
\end{align*}
Thus, $v^1=\begin{pmatrix}  \frac{s-1}{s} \\ \frac{(3-\sqrt{3})s^2+\sqrt{3}}{s(s+1)^2} \end{pmatrix}$. Let $v^2$ be such that 
\begin{align*}
 (v^2)^T \begin{pmatrix}  \frac{2}{\sqrt{3}} \frac{s+1}{s} \\ 0 \end{pmatrix}= 1 \quad\text{and}\quad
 (v^2)^T \begin{pmatrix} \frac{1}{\sqrt{3}} \frac{s^2+s-1}{s(s-1)} \\ \frac{1}{3} \frac{s+1}{s} \end{pmatrix}= 1.
\end{align*}
Thus, $v^2=\begin{pmatrix} \frac{\sqrt{3}}{2} \frac{s}{s+1} \\
\frac{1}{2} \frac{s^2-s-1}{s^2-1}
\end{pmatrix}$.
Let $v^3$ be such that 
\begin{align*}
 (v^3)^T \begin{pmatrix} \frac{1}{\sqrt{3}}\left( \frac{s+1}{s} \right) \\ \frac{1}{3} \frac{(s+1)^2}{s} \end{pmatrix}= 1 \quad\text{and}\quad
 (v^3)^T \begin{pmatrix} -\frac{1}{\sqrt{3}}\left( \frac{s+1}{s} \right) \\ \frac{1}{3} \frac{(s+1)^2}{s} \end{pmatrix}= 1.
\end{align*}
Thus, $v^3=\begin{pmatrix}  0 \\
\frac{3s}{(s+1)^2} \end{pmatrix}$. Finally, we conclude from the symmetries of $\frac{C^\circ - C^\circ}{2}$ that
\[
\left(\frac{C^\circ - C^\circ}{2} \right)^{\circ}= \conv\left(\left\{\begin{pmatrix} \pm \frac{s-1}{s} \\ \pm \frac{(3-\sqrt{3})s^2+\sqrt{3}}{s(s+1)^2} \end{pmatrix}, \begin{pmatrix} \pm \frac{\sqrt{3}}{2} \frac{s}{s+1} \\ \pm
\frac{1}{2} \frac{s^2-s-1}{s^2-1}
\end{pmatrix}, \begin{pmatrix}  0 \\
\pm \frac{3s}{(s+1)^2} \end{pmatrix}\right\}\right).
\]
The next thing we do is to compute scaling factors $\mu_1,\mu_2,\mu_3$ with respect to the different types of vertices of $\left(\frac{C^\circ - C^\circ}{2} \right)^{\circ}$ mapping them to the boundary of $\frac{C-C}{2}$. Using the geometry of the two sets the calculations below suffice.
For $\mu_1$ we have to solve
\[
\mu_1 \begin{pmatrix} \frac{s-1}{s} \\ \frac{(3-\sqrt{3})s^2+\sqrt{3}}{s(s+1)^2} \end{pmatrix} = \lambda \begin{pmatrix} \frac{\sqrt{3}}{4} (s-1) \\ \frac{3}{4} \end{pmatrix} + (1-\lambda) \begin{pmatrix} \frac{\sqrt{3}}{2} (s-1) \\ \frac{3}{4} (2-s) \end{pmatrix}.
\]
with $\lambda \in [0,1]$ and obtain
\[
\lambda=\frac{s^2+2s(\sqrt{3}-1)-3}{s^2+\sqrt{3}s-1} \quad \text{and} \quad 
\mu_1=\frac{ \frac{\sqrt{3}}{4} s(s+1)^2 }{s^2+\sqrt{3} s-1}. 
\]
For $\mu_2>1$ we obtain that either the equations 
\begin{align*}
 \mu_2 \begin{pmatrix} \frac{\sqrt{3}}{2}\frac{s}{s+1} \\
\frac{1}{2} \frac{s^2-s-1}{s^2-1}
\end{pmatrix} &= \lambda \begin{pmatrix} \frac{\sqrt{3}}{2} (s-1) \\ \frac{3}{4} (2-s) \end{pmatrix}+(1-\lambda) \begin{pmatrix} \frac{\sqrt{3}}{2} (s-1) \\ -\frac{3}{4} (2-s) \end{pmatrix} 
\intertext{or the equations}
 \mu_2 \begin{pmatrix} \frac{\sqrt{3}}{2} \frac{s}{s+1} \\ \frac{1}{2} \frac{s^2-s-1}{s^2-1} \end{pmatrix} &= \lambda
\begin{pmatrix} \frac{\sqrt{3}}{2} (s-1) \\
\frac{3}{4} 
\end{pmatrix}+(1-\lambda) \begin{pmatrix} \frac{\sqrt{3}}{2} (s-1) \\ \frac{3}{4} (2-s) \end{pmatrix}
\end{align*}
have to be fulfilled for some $\lambda \in [0,1]$.
It turns out that for $s \le \frac{4+\sqrt{26}}{5}$ holds the first system from which we obtain $\lambda= \frac{-s^2+4s-2}{6s(2-s)}$ and $\mu_2=\frac{s^2-1}{s}$.
In case $\frac{4+\sqrt{26}}{5} \le s \le 2$ the second system is the right one, which gives us 
$\lambda= \frac{5s^2-8s-2}{-2s^2+8s-1}$ and $\mu_2=\frac{3(3s-7)}{2s^2-8s+1}$.

For the last factor, $\mu_3>1$, we have to solve 
\[
\mu_3 \begin{pmatrix} 0 \\ \frac{3s}{(s+1)^2} \end{pmatrix} =\lambda \begin{pmatrix} -\frac{\sqrt{3}}{4} (s-1) \\ \frac{3}{4} \end{pmatrix}+ (1-\lambda) \begin{pmatrix} \frac{\sqrt{3}}{4} (s-1) \\ \frac{3}{4} \end{pmatrix}  
\]
for some  $\lambda \in [0,1]$, which leads to
$\lambda= \frac12$ and $\mu_3 = \frac{(s+1)^2}{4s}$. 

Finally, let $\beta(s) \leq 1$ be such that $\left(\frac{C^\circ - C^\circ}{2} \right)^{\circ} \subset^{opt} \beta(s) \cdot \frac{C-C}{2}$. Then, due to the symmetries of $C$ and since $\mu_1 > \max\{\mu_2,\mu_3\}$,
we obtain
\begin{equation*}
  \beta(s) = \max \left\{ \frac{1}{\mu_1}, \frac{1}{\mu_2}, \frac{1}{\mu_3} \right\} = \max \left\{ \frac{1}{\mu_2}, \frac{1}{\mu_3} \right\} = 
    \begin{cases}
      \frac{s}{s^2-1} & \text{if} \quad s \in [\frac{1+\sqrt{5}}{2},\frac53],\\
      \frac{4s}{(s+1)^2}, & \text{if} \quad s \in [\frac53,2].
    \end{cases}       
\end{equation*}
\end{enumerate}
Here one should note that we already know from (i) that $\beta_2 \ge \beta_1 \ge  \frac{4s}{(s+1)^2}$, i.e.~the $\frac{4s}{(s+1)^2}$ part for $s \in [\frac53,2]$ of the above calculation provides essentially no new information.

\end{proof}

Let us remark that for any asymmetry value $s \in [1,2]$ and any factor $\alpha(s) \in [\alpha_1(s),\alpha_2(s)]$ or $\beta(s) \in [\beta_1(s),\beta_2(s)]$ from Theorem \ref{thm:small_asym_no_improve}, respectively, there exists a Minkowski centered set 
$C$ with asymmetry $s(C)=s$ such that $C \cap (-C) \subset^{opt} \alpha(s) \conv(C \cup (-C))$  or $(\frac{C^\circ - C^\circ}{2})^\circ \subset^{opt} \beta(s) \frac{C-C}{2}$, respectively.

Finally, we would also like to mention that similar to the upper bounds on $\alpha(s)$ and $\beta(s)$ for $s(C)$ close to $n$ from
Theorem \ref{thm:small_asym_no_improve} (namely, 
$\alpha_2(s) \le  \psi \frac{n}{n+1}$ and $\beta_2(s) \le \mu \psi \frac{n(n+2)}{(n+1)^2}$) 
one may use the ideas from the proof of Theorem \ref{thm:minMax_mean_improved} to derive also lower bounds on $\alpha(s)$ and $\beta(s)$ for $s(C)$ close to $n$, i.e., $\alpha_1(s) \ge  f_1(s)$ and $\beta_1(s) \ge \mu f_2(s)$ for some continuous functions $f_1$, $f_2$ fulfilling $f_1(n)=\frac{n}{n+1}$ and $f_2(n)=\frac{n(n+1)}{(n+1)^2}$.

\section*{Appendix} \label{sec:app}

The \cemph{circumradius} of $K$ w.r.t.~the gauge body $C$ is defined as 
\[R(K,C) = \min \{\rho \ge 0 : K \subset \rho C +t, \ t \in \R^n\}.\] 

Surprisingly, the definition of a diameter with respect to a (possibly) non-symmetric gauge body $C \in \K^n$ (with $0 \in \inter(C)$) is not unified. While in \cite{Le} 
it is defined as 
\[
D_{\max}(K,C) = \max_{x,y \in K} \|x-y\|_C, 
\]
which we call the \cemph{maximal diameter}, and which at first view is the most natural generalization of a diameter for non-symmetric gauges; others (see c.f.~\cite{DGK}) preferred and partly argued to choose the following definition
of the \cemph{diameter} of $K$ w.r.t.~$C$:
\[
D(K,C) = 2 \max_{x,y \in K} R(\{x,y\},C). 
\]
The latter definition allows to see the diameter as a best 2-point approximation of the circumradius of the whole set $K$. Another advantage of it is that it is translation invariant in both arguments. In contrast, for the maximal diameter 
choosing $C$ with $0$ close to the boundary of $C$, the circumradius-diameter ratio may get arbitrarily small.

However, the choice of a definition should always fit its desired properties. For instance, if choosing an asymmetric gauge body is motivated by the desire to measure the distance from $x$ to $y$ different than that from $y$ to $x$, the latter should possibly be reflected in the length measurements (instead of measuring the length of the segment $[x,y]$ independently of its direction). 
Thus there may be applications where we would prefer to measure the distance from $x$ to $y$ by $\|x-y\|_C$, which then would lead us to the maximal diameter. 

And, this is part of our motivation for the investigation above, one can see that 
\[
D_{\max}(K,C) = D(K, C \cap (-C)), \quad \text{while} \quad D(K,C) = D\left(K, \frac{C-C}2\right).
\]
Moreover, if $C$ is Minkowski centered, the results above show us, that we can bound those diameters in terms of the other and therefore also the circumradius-diameter ratio for the maximal diameter.

Finally, it is easy to see that there are also well motivated definitions of lengths of segments or directional breadths w.r.t.~a given gauge $C$ that lead to diameters that depend on the harmonic mean $\left(\frac{C^\circ - C^\circ}2\right)^\circ$ or the maximum $\conv(C \cup (-C))$.

\end{document}